\documentclass{article}
\usepackage{ulem}

\usepackage{amssymb,amsmath,graphicx,amsthm,mathrsfs,bm}
\usepackage[textsize=footnotesize,color=DarkGreen!40]{todonotes}

\usepackage[colorlinks=true]{hyperref}
\usepackage{pdfsync,color}
\allowdisplaybreaks[4]
\numberwithin{equation}{section}
\newtheorem{theorem}{Theorem}[section]
\newtheorem{Theorem}{Theorem}[section]
\newtheorem{Lemma}[theorem]{Lemma}

\newtheorem{Proposition}[theorem]{Proposition}
\theoremstyle{definition}

\theoremstyle{definition}
\newtheorem{remark}[theorem]{Remark}
\numberwithin{equation}{section}

\newcommand{\lc}
{\mathrel{\raise2pt\hbox{${\mathop<\limits_{\raise1pt\hbox
{\mbox{$\sim$}}}}$}}}

\newcommand{\gc}
{\mathrel{\raise2pt\hbox{${\mathop>\limits_{\raise1pt\hbox{\mbox{$\sim$}}}}$}}}

\newcommand{\ec}
{\mathrel{\raise2pt\hbox{${\mathop=\limits_{\raise1pt\hbox{\mbox{$\sim$}}}}$}}}

\def\bb{\begin{equation}} \def\ee{\end{equation}}

\def\beqn{\begin{eqnarray}}  \def\eqn{\end{eqnarray}}

\def\beqnx{\begin{eqnarray*}} \def\eqnx{\end{eqnarray*}}

\def\bn{\begin{enumerate}} \def\en{\end{enumerate}}

\def\bd{\begin{description}} \def\ed{\end{description}}

\def\label{\label}

\renewcommand{\leq}{\leqslant}
\renewcommand{\geq}{\geqslant}

\begin{document}

\title{Quantitative unique continuation for parabolic equations with Neumann boundary conditions}

\author{Yueliang Duan, \thanks{Department of Mathematics, Shantou University, Shantou 515063, China;
e-mail: ylduan@stu.edu.cn.}\quad
Lijuan Wang, \thanks{School of
Mathematics and Statistics, Computational Science Hubei Key Laboratory,
Wuhan University, Wuhan 430072, China;
e-mail: ljwang.math@whu.edu.cn. This work was supported by the
National Natural Science Foundation of China under grant 12171377.}
\quad Can Zhang
\thanks{School of Mathematics and Statistics, Wuhan University, Wuhan 430072, China;
e-mail: canzhang@whu.edu.cn. This work was partially supported by NSCF under the grant 11971363.}}

\maketitle

\noindent\textbf{Abstract.}  In this paper, we establish a globally quantitative estimate of unique continuation at one time point for solutions of parabolic equations with Neumann boundary conditions in bounded domains.
Our proof is mainly based on Carleman commutator estimates and a global frequency function argument, which is motivated from a recent work \cite{Buffe}.  As an application,  we obtain an observability inequality from measurable sets in time for all solutions of the above equations.

\medskip

\medskip


\section{Introduction}

Let $T>0$ and let  $\Omega\subset \mathbb{R}^N (N\geq1)$ be a bounded connected open set with
 boundary $\partial\Omega$ of class $C^\infty$.  Let the principal part $A(\cdot)$ be a $N\times N$ symmetric matrix with $C^{2}(\overline{\Omega})$ coefficients and satisfy the uniform ellipticity condition, i.e.,  there is a constant $\lambda\geq1$ so that
\begin{equation}\label{yu-11-28-2}
    \lambda^{-1}|\xi|^2\leq A(x)\xi\cdot\xi \leq \lambda|\xi|^2\;\;\text{for all}\;\;x\in\mathbb R^N\;\;\text{and}\;\;\xi\in\mathbb{R}^N.
\end{equation}
We consider the following linear parabolic equation with the homogeneous Neumann boundary condition:
\begin{equation}\label{1.1}
\left\{ \begin{array}{lll}
\partial_{t}u-\mathrm{div}(A\nabla u)+B\cdot\nabla u+au=0&\mathrm{in}\  \Omega\times (0,T),\\
A\nabla u\cdot \overrightarrow{n}=0&\mathrm{on}\  \partial\Omega\times (0,T),\\
u(\cdot,0)\in L^2(\Omega).\\
\end{array}\right.\end{equation}
Here, the potentials $a\in L^\infty(\Omega\times (0,T))$, $B\in (L^\infty(\Omega\times (0,T)))^{N}$, and $\overrightarrow{n}$ is the unit outward normal vector to $\partial\Omega$. According to \cite[Theorem 10.9]{Brezis} and \cite[Theorem 4.3]{Barbu}, the equation \eqref{1.1} has a unique solution $u\in L^2(0,T;H^1(\Omega))\cap C([0,T]; L^2(\Omega))$. Moveover, for each $\delta \in (0, T)$, $u\in C([\delta,T]; H^1(\Omega))$.

Throughout the paper, we denote the usual inner product and norm in $(L^2(\Omega))^k (k\geq 1)$  by $\langle\cdot,\cdot\rangle$ and $\|\cdot\|$, respectively; $\| a\|_{\infty} := \| a\|_{L^{\infty}(\Omega \times (0,T))}$; $\|B\|_{\infty} := \| B\|_{(L^{\infty}(\Omega \times (0,T)))^N}$; $C(\cdot )$ denotes a generic positive constant depending on what are enclosed in the brackets. We shall occasionally
use the sum index convention.

The main theorems of this paper concerning the quantitative estimate of unique continuation, as well as the observability inequality, for all solutions of \eqref{1.1} can be stated as follows.

\begin{Theorem}\label{jiudu4}
Let $\widetilde{\omega}$ be a non-empty open subset in $\Omega$. Then, there are constants  $\mathcal{K}:=\mathcal{K}(A, \Omega, \widetilde{\omega})>0$ and $\beta:=\beta(A, \Omega, \widetilde{\omega})\in(0,1)$ so that for any $t\in(0,T]$ and any $u(\cdot,0)\in L^{2}(\Omega)$, the corresponding solution to \eqref{1.1} satisfies
\begin{eqnarray*}
\int_{\Omega}|u(x,t)|^{2}\mathrm{d}x\leq e^{\mathcal{K}\left[1+\frac{1}{t}+\|a\|_{\infty}^{2/3}+\|B\|_{\infty}^{2}+t\left(\|a\|_{\infty}+\|B\|^{2}_{\infty}\right)\right]}\left(\int_{\Omega}|u(x,0)|^{2}\mathrm{d}x\right)^{\beta}\left(\int_{\widetilde{\omega}} |u(x,t)|^{2}\mathrm{d}x\right)^{1-\beta}.
\end{eqnarray*}
\end{Theorem}

\medskip

\begin{Theorem}\label{Thm1}
Let $E$ be a subset of positive Lebesgue measure in $(0,T)$, and let $\omega$ be a non-empty open subset of $\Omega$. Then there exist two constants $\widetilde{\mathcal{K}}_1:=\widetilde{\mathcal{K}}_1(E, A, \Omega, \omega)>0$
and $\widetilde{\mathcal{K}}_2:=\widetilde{\mathcal{K}}_2(A,\Omega,\omega)>0$ so that for any $u(\cdot,0)\in L^{2}(\Omega)$, the corresponding  solution of \eqref{1.1} satisfies
	\begin{equation*}
	\left(\int_{\Omega}|u(x,T)|^2\mathrm{d}x\right)^{1/2}\leq e^{\widetilde{\mathcal{K}}_1} e^{\widetilde{\mathcal{K}}_2\left[\|a\|_{\infty}^{2/3}+\|B\|_{\infty}^{2}+T\left(\|a\|_{\infty}+\|B\|^{2}_{\infty}\right)\right]}\int_{\omega\times E}|u(x,t)|\mathrm{d}x\mathrm{d}t.
\end{equation*}
\end{Theorem}

\medskip

Several remarks are given in order.

\begin{remark}
When the principal part $A(\cdot)\equiv I$ (the $N\times N$ order identity matrix) and the potential $B=0$ in the equation \eqref{1.1},  a similar result to Theorem \ref{jiudu4}  has been recently established in \cite{Buffe}. Here we extend the result to the more general case by adapting the method used in \cite{Buffe}.
\end{remark}

\begin{remark}
In the particular case that $E=(0,T)$,  the constant $\widetilde{\mathcal{K}}_1(E, A, \Omega, \omega)$ appearing  in Theorem \ref{Thm1}  could be taken the form of $\widetilde{\mathcal{K}}_1(A, \Omega, \omega)/T$. Indeed, the observability constant obtained here has the same optimal dependence on the $L^\infty$-norm of potentials $a$ and $B$ as those in the nice work \cite{xz}.
\end{remark}

\begin{remark}
For the more general case that  the principle part $A(\cdot)$ in the equation \eqref{1.1} depend on both the time and spatial variables, similar results could be obtained by using the method of the proof we developed here, as well as the argument in  \cite[Theorem 4.1]{Claude}.
\end{remark}

The interpolation inequality at one time point as stated in Theorem \ref{jiudu4}  is a globally
quantitative version of unique continuation for solutions of parabolic equations.
The investigate of unique continuation properties  for solutions of parabolic-type partial differential equations has a very long history. Giving an exhaustive bibliography on the subject is by far beyond the scope of the present paper, but the reader can consult the survey article \cite{Ves} and the literature cited there.

It is worthing to mention that the same type of interpolation inequality as in Theorem \ref{jiudu4} has been obtained in \cite{Claude, Wang0,Phung-Wang,Phung-Wang-Zhang} for solutions of parabolic equations with the zero Dirichlet boundary condition, by using a (local)-parabolic-type frequency function method developed in \cite{iaza}.  To the best of our knowledge, the arguments in the local nature of these works are not valid in our present case because of different boundary conditions.

The observability inequality on measurable subsets for parabolic
equations was previously studied in a large number of publications. When $E$ is the whole time interval and the observation region $\omega$ is a non-empty open subset, we refer the reader to the classical monograph \cite{FI}. The approach is mainly based on the method of a globally Carleman estimate, which involves with an exponential weight function.  When $E$ is only a subset of positive Lebesgue measure in the time interval
and the observation region $\omega$ is a non-empty open subset, we refer the reader to
a series of recent works \cite{Apraiz-Escauriaza-Wang-Zhang,Phung-Wang, Phung-Wang-Zhang,wang-zhang1} for the observability inequality for parabolic-type evolution systems.

We would like to stress that the observability estimate from measurable
sets in the time variable established as in Theorem \ref{Thm1} has several applications
in control theory.  In particular, it implies bang-bang properties of minimal norm and minimal time optimal control problems (see, for instance, \cite{Phung-Wang}).

The first effort of this paper is to prove a quantitative estimate of unique continuation for parabolic equations by using a frequency function method, which is borrowed  from \cite{Buffe}. Secondly, we utilize  a telescoping method to prove an observability inequality for the same equations.

The rest of this paper is organized as follows. In Section \ref{sc2}, we show several auxiliary lemmas, whose proofs are provided in Appendix.  Sections \ref{kaodu3}
 and  \ref{finalproof} prove Theorems~\ref{jiudu4} and \ref{Thm1}, respectively.

\section{Preliminaries}\label{sc2}

First of all, we give two standard energy estimates for solutions of \eqref{1.1}. For the sake of completeness we provide their detailed proofs in the appendix.

\begin{Lemma}\label{lemma-1.1}
For all $t\in [0,T]$, the solution $u$ of \eqref{1.1} satisfies
\begin{eqnarray}\label{1.2}
\int_{\Omega}| u(x,t)|^{2}\mathrm{d}x+\lambda^{-1}\int_{0}^{t}\int_{\Omega}|\nabla u|^{2}\mathrm{d}x\mathrm{d}s
\leq  e^{t\left(2\|a\|_{\infty}+\lambda\|B\|^{2}_{\infty}\right)
}\int_{\Omega}|u(x,0)|^{2}\mathrm{d}x.
\end{eqnarray}
\end{Lemma}
\begin{Lemma}\label{lemma-1.2}
For all $t\in(0,T]$, the solution $u$ of \eqref{1.1} satisfies
\begin{eqnarray}\label{1.3}
\int_{\Omega}|\nabla u(x,t)|^{2}\mathrm{d}x\leq\frac{2\lambda^{3}}{t} e^{t\left[3\|a\|_{\infty}+2\lambda\|B\|^{2}_{\infty}\right]}\int_{\Omega}|u(x,0)|^{2}\mathrm{d}x.
\end{eqnarray}
\end{Lemma}

The next result is concerned with  some identities linked to the Carleman commutator.
\begin{Proposition}\label{lemma-2}
Let $$\Phi(x,t):=\frac{s\varphi(x)}{\Gamma(t)}, \ s>0, \ \Gamma(t):=T-t+h, \ h>0, \ \varphi\in C^{\infty}(\overline{\Omega}).$$
For any $f\in H^2(\Omega)$, let

\begin{equation}\label{absd8}
\left\{
\begin{array}{l}
\mathcal{A}_{\varphi}f:=-A\nabla\Phi\cdot\nabla f-\frac{1}{2}\mathrm{div}(A\nabla\Phi)f,\\
\mathcal{S}_{\varphi}f:=-\mathrm{div}(A\nabla f)-\eta f, \ where\ \eta:=\frac{1}{2}\partial_{t}\Phi+\frac{1}{4}A\nabla\Phi\cdot\nabla\Phi,\\
\mathcal{S'}_{\varphi}f:=-\partial_{t}\eta f.
\end{array}
\right.
\end{equation}
Then, we have 
\begin{description}
\item[$(i)$] $$\int_{\Omega}\mathcal{A}_{\varphi}f f\mathrm{d}x=-\frac{1}{2}\int_{\partial\Omega}A\nabla\Phi\cdot \overrightarrow{n} f^2 \mathrm{d}S;$$

\item[$(ii)$] $$\int_{\Omega}\mathcal{S}_{\varphi}f f\mathrm{d}x=\int_{\Omega}A\nabla f\cdot\nabla f\mathrm{d}x-\int_{\Omega}\eta f^2\mathrm{d}x-\int_{\partial\Omega}A\nabla f\cdot \overrightarrow{n} f \mathrm{d}S;$$

\item[$(iii)$] \begin{eqnarray*}
&&\int_{\Omega}\mathcal{S'}_{\varphi}f f\mathrm{d}x+2\int_{\Omega}\mathcal{S}_{\varphi}f \mathcal{A}_{\varphi}f\mathrm{d}x\\
&=&2\int_{\partial\Omega}(A\nabla f\cdot \overrightarrow{n})( A\nabla\Phi\cdot\nabla f)\mathrm{d}S-\int_{\partial\Omega}(A\nabla \Phi \cdot \overrightarrow{n}) (A\nabla f\cdot\nabla f)\mathrm{d}S\\
&&+\int_{\partial\Omega}(A\nabla f\cdot \overrightarrow{n})\mathrm{div}(A\nabla \Phi) f\mathrm{d}S+\int_{\partial\Omega}(A\nabla \Phi \cdot \overrightarrow{n})\eta f^2\mathrm{d}S\\
&&-2\int_{\Omega}A_{ij}\partial_{x_{j}}f\partial_{x_{i}}A_{kl}\partial_{x_{l}}f\partial_{x_{k}}\Phi \mathrm{d}x+\int_{\Omega}\partial_{x_{l}}A_{ij}\partial_{x_{j}}fA_{kl}\partial_{x_{i}}f\partial_{x_{k}}\Phi \mathrm{d}x  \\
&&-2\int_{\Omega}A\nabla^2 \Phi A\nabla f\cdot \nabla f \mathrm{d}x-\int_{\Omega}A\nabla f\cdot\nabla\left[\mathrm{div}(A\nabla \Phi)\right]f\mathrm{d}x\\
&&-\frac{2}{\Gamma}\int_{\Omega} \left(\eta+\frac{1}{4}A\nabla\Phi\cdot\nabla\Phi+\frac{s}{8}A_{ij}\partial_{x_{j}}\varphi\partial_{x_{i}}A_{kl}\partial_{x_{l}}\Phi\partial_{x_{k}}\Phi+\frac{s}{4}A\nabla^{2}\varphi A\nabla\Phi\cdot\nabla\Phi\right)  f^2 \mathrm{d}x.
\end{eqnarray*}
\end{description}
\end{Proposition}
\begin{proof}~By integrations by parts, we can easily check $(i)$ and $(ii)$. We next turn to the proof of $(iii)$. According to
the definitions of $\mathcal{S}_{\varphi}f$ and $\mathcal{A}_{\varphi}f,$ it is clear that

\begin{eqnarray}\label{absd1}
2\langle\mathcal{S}_{\varphi}f,  \mathcal{A}_{\varphi}f\rangle=\int_{\Omega}\left[\mathrm{div}(A\nabla f)+\eta f\right]\left[2A\nabla\Phi\cdot\nabla f+\mathrm{div}(A\nabla\Phi)f\right]\mathrm{d}x.
\end{eqnarray}

Firstly, by integrations by parts, we have that
\begin{eqnarray*}
&&2\int_{\Omega}\mathrm{div}(A\nabla f)A\nabla\Phi\cdot\nabla f\mathrm{d}x\\
&=&2\int_{\partial\Omega}(A\nabla f\cdot \overrightarrow{n}) (A\nabla\Phi\cdot\nabla f)\mathrm{d}S-2\int_{\Omega}A\nabla f\cdot \nabla(A\nabla\Phi\cdot\nabla f)\mathrm{d}x
\\
&=& 2\int_{\partial\Omega}(A\nabla f\cdot \overrightarrow{n})( A\nabla\Phi\cdot\nabla f)\mathrm{d}S-2\int_{\Omega}A_{ij}\partial_{x_{j}}f\partial_{x_{i}}A_{kl}\partial_{x_{l}}f\partial_{x_{k}}\Phi \mathrm{d}x \\
&&-2\int_{\Omega}A\nabla^2 \Phi A\nabla f\cdot \nabla f \mathrm{d}x-2\int_{\Omega}A\nabla^2 f A\nabla f\cdot \nabla \Phi \mathrm{d}x,
\end{eqnarray*}
and
\begin{eqnarray*}
&&-2\int_{\Omega}A\nabla^2 f A\nabla f\cdot \nabla \Phi \mathrm{d}x\\
&=&-\int_{\partial\Omega}(A\nabla \Phi \cdot \overrightarrow{n}) (A\nabla f\cdot\nabla f)\mathrm{d}S+\int_{\Omega}\partial_{x_{l}}A_{ij}\partial_{x_{j}}fA_{kl}\partial_{x_{i}}f\partial_{x_{k}}\Phi \mathrm{d}x \\
&&+\int_{\Omega}\mathrm{div}(A\nabla \Phi)A\nabla f\cdot \nabla f\mathrm{d}x.
\end{eqnarray*}
The above two equalities imply that
\begin{eqnarray}\label{absd2}
&&2\int_{\Omega}\mathrm{div}(A\nabla f)A\nabla\Phi\cdot\nabla f\mathrm{d}x\nonumber\\
&=& 2\int_{\partial\Omega}(A\nabla f\cdot \overrightarrow{n})( A\nabla\Phi\cdot\nabla f)\mathrm{d}S-\int_{\partial\Omega}(A\nabla \Phi \cdot \overrightarrow{n}) (A\nabla f\cdot\nabla f)\mathrm{d}S\nonumber\\
&&-2\int_{\Omega}A_{ij}\partial_{x_{j}}f\partial_{x_{i}}A_{kl}\partial_{x_{l}}f\partial_{x_{k}}\Phi \mathrm{d}x+\int_{\Omega}\partial_{x_{l}}A_{ij}\partial_{x_{j}}fA_{kl}\partial_{x_{i}}f\partial_{x_{k}}\Phi \mathrm{d}x  \nonumber\\
&&-2\int_{\Omega}A\nabla^2 \Phi A\nabla f\cdot \nabla f \mathrm{d}x+\int_{\Omega}\mathrm{div}(A\nabla \Phi)A\nabla f\cdot \nabla f\mathrm{d}x.
\end{eqnarray}

Secondly, by integrations by parts, we get that
\begin{eqnarray}\label{absd3}
&&\int_{\Omega}\mathrm{div}(A\nabla f)\mathrm{div}(A\nabla \Phi) f\mathrm{d}x\nonumber\\
&=& \int_{\partial\Omega}(A\nabla f\cdot \overrightarrow{n})\mathrm{div}(A\nabla \Phi) f\mathrm{d}S-\int_{\Omega}A\nabla f\cdot\nabla\left[\mathrm{div}(A\nabla \Phi) f\right]\mathrm{d}x\nonumber\\
&=& \int_{\partial\Omega}(A\nabla f\cdot \overrightarrow{n})\mathrm{div}(A\nabla \Phi) f\mathrm{d}S-\int_{\Omega}A\nabla f\cdot\nabla\left[\mathrm{div}(A\nabla \Phi)\right]f\mathrm{d}x\nonumber\\
&&-\int_{\Omega}\mathrm{div}(A\nabla \Phi)A\nabla f\cdot\nabla f \mathrm{d}x,
\end{eqnarray}
and
\begin{eqnarray}\label{absd4}
2\int_{\Omega}A\nabla\Phi\cdot\nabla f \eta f \mathrm{d}x
&=&\int_{\partial\Omega}(A\nabla \Phi \cdot \overrightarrow{n})\eta f^2\mathrm{d}S-\int_{\Omega} A\nabla\Phi\cdot\nabla \eta  f^2 \mathrm{d}x\nonumber\\
&&-\int_{\Omega}\eta \mathrm{div}(A\nabla \Phi)f^2\mathrm{d}x.
\end{eqnarray}
It follows from the definition of $\mathcal{S'}_{\varphi}f$ and \eqref{absd1}-\eqref{absd4} that
\begin{eqnarray}\label{absd5}
&&\int_{\Omega}\mathcal{S'}_{\varphi}f f\mathrm{d}x+2\int_{\Omega}\mathcal{S}_{\varphi}f \mathcal{A}_{\varphi}f\mathrm{d}x\nonumber\\
&=& 2\int_{\partial\Omega}(A\nabla f\cdot \overrightarrow{n})( A\nabla\Phi\cdot\nabla f)\mathrm{d}S-\int_{\partial\Omega}(A\nabla \Phi \cdot \overrightarrow{n}) (A\nabla f\cdot\nabla f)\mathrm{d}S\nonumber\\
&&+ \int_{\partial\Omega}(A\nabla f\cdot \overrightarrow{n})\mathrm{div}(A\nabla \Phi) f\mathrm{d}S+\int_{\partial\Omega}(A\nabla \Phi \cdot \overrightarrow{n})\eta f^2\mathrm{d}S\nonumber\\
&&-2\int_{\Omega}A_{ij}\partial_{x_{j}}f\partial_{x_{i}}A_{kl}\partial_{x_{l}}f\partial_{x_{k}}\Phi \mathrm{d}x+\int_{\Omega}\partial_{x_{l}}A_{ij}\partial_{x_{j}}fA_{kl}\partial_{x_{i}}f\partial_{x_{k}}\Phi \mathrm{d}x  \nonumber\\
&&-2\int_{\Omega}A\nabla^2 \Phi A\nabla f\cdot \nabla f \mathrm{d}x-\int_{\Omega}A\nabla f\cdot\nabla\left[\mathrm{div}(A\nabla \Phi)\right]f\mathrm{d}x\nonumber\\
&&-\int_{\Omega} (\partial_{t}\eta+A\nabla\Phi\cdot\nabla \eta)  f^2 \mathrm{d}x.
\end{eqnarray}

Finally, using $\partial_{t}\Phi=\frac{1}{\Gamma}\Phi$ and $\partial_{t}^2\Phi=\frac{2}{\Gamma}\partial_{t}\Phi$, we obtain that
\begin{eqnarray*}
&&-\partial_{t}\eta-A\nabla\eta\cdot\nabla\Phi\\
&=&-\frac{1}{2}\partial_{t}^2\Phi-A\nabla\Phi\cdot\nabla\Phi_{t}-\frac{1}{4}A\nabla(A\nabla\Phi\cdot\nabla\Phi)\cdot\nabla\Phi\\
&=&-\frac{1}{\Gamma}(2\eta-\frac{1}{2}A\nabla\Phi\cdot\nabla\Phi)-\frac{1}{\Gamma}A\nabla\Phi\cdot\nabla\Phi-\frac{1}{4}A_{ij}\partial_{x_{j}}\Phi\partial_{x_{i}}A_{kl}\partial_{x_{l}}\Phi\partial_{x_{k}}\Phi-\frac{1}{2}A\nabla^{2}\Phi A\nabla\Phi\cdot\nabla\Phi\\
&=&-\frac{2}{\Gamma}\eta-\frac{1}{2\Gamma}A\nabla\Phi\cdot\nabla\Phi-\frac{s}{4\Gamma}A_{ij}\partial_{x_{j}}\varphi\partial_{x_{i}}A_{kl}\partial_{x_{l}}\Phi\partial_{x_{k}}\Phi-\frac{s}{2\Gamma}A\nabla^{2}\varphi A\nabla\Phi\cdot\nabla\Phi.
\end{eqnarray*}
This, along with \eqref{absd5}, yields $(iii)$.
\end{proof}

The following three results will be useful in the proof of Theorem \ref{jiudu4}.
\begin{Proposition}\label{lemma-1234}
Let $h>0, T>0$ and $F_{1}(\cdot), F_{2}(\cdot)\in C([0,T]).$ Consider two positive functions $y(\cdot), N(\cdot)\in C^1([0,T])$ such that
\begin{equation}\label{KKKK}
\left\{
\begin{array}{l}
\left|\frac{1}{2}y'(t)+N(t)y(t)\right|\leq \left(\frac{1}{2}N(t)+\frac{S_{0}}{T-t+h}+S_{1}\right)y(t)+F_{1}(t)y(t),\\
\\
N'(t)\leq \left(\frac{1+S_{0}}{T-t+h}+S_{1}\right)N(t)+F_{2}(t),
\end{array}
\right.
\end{equation}
where $S_0, S_{1}\geq0.$ Then for any $0\leq t_{1}< t_{2}<t_{3}\leq T,$ one has
$$y(t_2)^{1+M}\leq y(t_3)y(t_1)^M e^{D} \left(\frac{T-t_{1}+h}{T-t_{3}+h}\right)^{3S_{0}(1+M)}, \;\;\ \forall \ M\geq M_{0},$$
with
\begin{equation}\label{12-26-1}
M_{0}:=3\frac{\int^{t_{3}}_{t_{2}}\frac{e^{tS_{1}}}{(T-t+h)^{1+S_{0}}}\mathrm{d}t}{\int^{t_{2}}_{t_{1}}\frac{e^{tS_{1}}}{(T-t+h)^{1+S_{0}}}
\mathrm{d}t}
\end{equation}
and
$$D:=3(M+1)\left[\left(\int_{t_{1}}^{t_{3}}|F_{2}|\mathrm{d}t+S_{1}\right)(t_{3}-t_{1})+\int_{t_{1}}^{t_{3}}|F_{1}|\mathrm{d}t\right].$$
\end{Proposition}

The proof of Proposition~\ref{lemma-1234} is a slight modification for that of Lemma 4.3 in \cite{Claude}.
For the sake of completeness, we give its detailed proof below.

\begin{proof}~By the second inequality of (\ref{KKKK}), we have that
\begin{equation}\label{KKKK1}
\left[(T-t+h)^{1+S_0}e^{-tS_{1}}N(t)\right]'\leq (T-t+h)^{1+S_0}e^{-tS_{1}}F_{2}(t).
\end{equation}
On one hand, integrating (\ref{KKKK1}) over $(t, t_2)$ with $t\in(t_{1}, t_{2})$, we obtain that
\begin{equation}\label{KKKKh1}
\left(\frac{T-t_{2}+h}{T-t+h}\right)^{1+S_{0}}e^{-S_{1}(t_2-t)}N(t_2)-\int_{t_{1}}^{t_{2}}|F_{2}(s)|\mathrm{d}s\leq N(t).
\end{equation}
It follows from the first inequality of (\ref{KKKK}) and (\ref{KKKKh1}) that
$$y'(t)+\left[\left(\frac{T-t_{2}+h}{T-t+h}\right)^{1+S_{0}}e^{-S_{1}(t_2-t)}N(t_2)-\frac{2S_{0}}{T-t+h}
-2S_{1}-2\int_{t_{1}}^{t_{2}}|F_{2}(s)|\mathrm{d}s-2|F_{1}(t)|\right]y(t)\leq0,$$
where $t\in [t_1,t_2]$. This implies that
\begin{equation}\label{KKKK2}
\begin{array}{lll}
&&e^{N(t_2)\displaystyle{}\int_{t_{1}}^{t_{2}}\left(\frac{T-t_{2}+h}{T-t+h}\right)^{1+S_{0}}e^{-S_{1}(t_2-t)}\mathrm{d}t}\\
&\leq& \displaystyle{\frac{y(t_{1})}{y(t_{2})}}\left(\displaystyle{\frac{T-t_{1}+h}{T-t_{2}+h}}\right)^{2S_{0}}e^{2(t_{2}-t_{1})
\left(S_{1}+\displaystyle{}\int_{t_{1}}^{t_{2}}|F_{2}(s)|\mathrm{d}s\right)+\displaystyle{}2\int_{t_{1}}^{t_{2}}|F_{1}(s)|\mathrm{d}s}.
\end{array}
\end{equation}
On the other hand, integrating (\ref{KKKK1}) over $(t_{2}, t)$ with $t\in(t_{2}, t_{3})$, we get that
\begin{eqnarray}\label{KKKKh2}
N(t)\leq e^{S_1(t-t_2)}\left(\frac{T-t_{2}+h}{T-t+h}\right)^{1+S_{0}}\left(N(t_{2})+\displaystyle{}\int_{t_{2}}^{t_{3}}|F_{2}(s)|\mathrm{d}s\right).
\end{eqnarray}
It follows from the first inequality of (\ref{KKKK}) and (\ref{KKKKh2}) that
$$y'(t)+3\left[\left(\frac{T-t_{2}+h}{T-t+h}\right)^{1+S_{0}}e^{S_{1}(t-t_2)}\left(N(t_2)
+\int_{t_{2}}^{t_{3}}|F_{2}(s)|\mathrm{d}s\right)+\frac{S_{0}}{T-t+h}+S_{1}+|F_{1}(t)|\right]y(t)\geq0,$$
where $t\in [t_2,t_3]$. This yields that
\begin{equation}\label{KKKK3}
\begin{array}{lll}
y(t_{2})&\leq& e^{3\left(N(t_2)+\displaystyle{}\int_{t_{2}}^{t_{3}}|F_{2}(s)|\mathrm{d}s\right)\displaystyle{}
\int_{t_{2}}^{t_{3}}\left(\frac{T-t_{2}+h}{T-t+h}\right)^{1+S_{0}}e^{S_{1}(t-t_2)}\mathrm{d}t}\\
&&
\times y(t_{3})\left(\displaystyle{\frac{T-t_{2}+h}{T-t_{3}+h}}\right)^{3S_{0}}e^{3S_{1}(t_{3}-t_{2})
+\displaystyle{}3\int_{t_{2}}^{t_{3}}|F_{1}(s)|\mathrm{d}s}.
\end{array}
\end{equation}

According to (\ref{KKKK3}) and (\ref{KKKK2}), it is clear that
\begin{eqnarray*}
y(t_{2})&\leq& y(t_{3})\left(\frac{y(t_{1})}{y(t_{2})}\left(\frac{T-t_{1}+h}{T-t_{2}+h}\right)^{2S_{0}} e^{2(t_{2}-t_{1})\left(S_{1}+\displaystyle{}\int_{t_{1}}^{t_{2}}|F_{2}(s)|\mathrm{d}s\right)
+\displaystyle{}2\int_{t_{1}}^{t_{2}}|F_{1}(s)|\mathrm{d}s}\right)^{M_{0}}\\
&&\times e^{\displaystyle{}M_{0}(t_{2}-t_{1})\int_{t_{2}}^{t_{3}}|F_{2}(s)|\mathrm{d}s} \left(\frac{T-t_{2}+h}{T-t_{3}+h}\right)^{3S_{0}}e^{3S_{1}(t_{3}-t_{2})+\displaystyle{}3\int_{t_{2}}^{t_{3}}|F_{1}(s)|\mathrm{d}s}
\end{eqnarray*}
and
\begin{equation*}
\frac{y(t_{1})}{y(t_{2})}\left(\displaystyle{\frac{T-t_{1}+h}{T-t_{2}+h}}\right)^{2S_{0}}e^{2(t_{2}-t_{1})
\left(S_{1}+\displaystyle{}\int_{t_{1}}^{t_{2}}|F_{2}(s)|\mathrm{d}s\right)+\displaystyle{}2\int_{t_{1}}^{t_{2}}|F_{1}(s)|\mathrm{d}s}\geq 1,
\end{equation*}
where $M_0$ is given by (\ref{12-26-1}). The above two inequalities imply that for any $M\geq M_{0}$,
\begin{equation}\label{12-26-2}
\begin{array}{lll}
y(t_{2})&\leq& y(t_{3})\left(\displaystyle{\frac{y(t_{1})}{y(t_{2})}}\displaystyle{\left(\frac{T-t_{1}+h}{T-t_{2}+h}\right)}^{2S_{0}} e^{2(t_{2}-t_{1})\left(S_{1}+\displaystyle{}\int_{t_{1}}^{t_{2}}|F_{2}(s)|\mathrm{d}s\right)
+\displaystyle{}2\int_{t_{1}}^{t_{2}}|F_{1}(s)|\mathrm{d}s}\right)^{M}\\
&&\times e^{\displaystyle{}M(t_{2}-t_{1})\int_{t_{2}}^{t_{3}}|F_{2}(s)|\mathrm{d}s} \displaystyle{\left(\frac{T-t_{2}+h}{T-t_{3}+h}\right)^{3S_{0}}}e^{3S_{1}(t_{3}-t_{2})+\displaystyle{}3\int_{t_{2}}^{t_{3}}|F_{1}(s)|\mathrm{d}s}.
\end{array}
\end{equation}
The result follows from (\ref{12-26-2}) immediately.\\

In summary, we finish the proof of Proposition~\ref{lemma-1234}.
\end{proof}

\begin{Proposition}[\cite{Buffe}]\label{lemma-absd1}
Let $\Omega\subset \mathbb{R}^N (N\geq1)$ be a bounded connected open set with boundary $\partial\Omega$ of class $C^{\infty}$, and let $\widetilde{\omega}$ be a nonempty open subset of $\Omega$. Then there exists a positive integer $d,$ some points $ p_1, p_2,\cdot\cdot\cdot,p_d \in \widetilde{\omega},$ and $(\psi_1, \psi_2,\cdot\cdot\cdot,\psi_d)\in (C^{\infty}(\overline{\Omega}))^d$ so that for all $1\leq i\leq d$,
\begin{description}
\item[$(i)$] $\psi_{i}>0$ in $\Omega$, $\psi_{i}=0$ on $\partial\Omega;$

\item[$(ii)$] $\{x\in \Omega: |\nabla\psi_{i}(x)|=0\}=\{p_{j}: 1\leq j\leq d\};$

\item[$(iii)$] $p_i$ is the unique global maximum of $\psi_{i}$;

\item[$(iv)$] the critical points of $\psi_{i}$ are nondegenerate;

\item[$(v)$] for any $1\leq j\leq d$, $\max_{\overline{\Omega}}\psi_{j}=\max_{\overline{\Omega}}\psi_{i}$.
\end{description}
\end{Proposition}
Let
\begin{equation}\label{12-22-add}
\left\{
\begin{array}{l}
\varphi_{i,1}:=\psi_{i}-\max_{\overline{\Omega}}\psi_{i}, 1\leq i\leq d,\\
\\
\varphi_{i,2}:=-\psi_{i}-\max_{\overline{\Omega}}\psi_{i}, 1\leq i\leq d.
\end{array}
\right.
\end{equation}
For all $1\leq i\leq d$, it is clear that
\begin{eqnarray}\label{absd6}
\varphi_{i,1}=\varphi_{i,2}\ \mathrm{on}\ \partial\Omega, \ \mathrm{and} \ \nabla\varphi_{i,1}+\nabla\varphi_{i,2}=0 \ \mathrm{on}\ \partial\Omega.
\end{eqnarray}
\begin{Proposition}[\cite{Buffe}]\label{lemma-1} Under the assumptions of Proposition~\ref{lemma-absd1}, there exist positive constants  $c_1 :=c_1(\Omega,\widetilde{\omega}),\cdots,c_{6} :=c_6(\Omega,\widetilde{\omega})$ so that for each $1\leq i\leq d$, there are
subsets
$\mathcal{B}_{i}, \mathcal{C}_{i}, \vartheta_{i}\subset \Omega$ with $\mathcal{B}_{i}\cap \mathcal{C}_{i}=\varnothing (1\leq i\leq d)$
and $\vartheta_{i}$ being a neighbourhood of $\partial\Omega$ satisfying
\begin{description}
\item[$(i)$]
 In $\mathcal{D}_{i}$, $$c_{1}|\nabla \varphi_{i,1}|^{2}\leq|\varphi_{i,1}|\leq c_{2}|\nabla \varphi_{i,1}|^{2},$$
where $\mathcal{D}_{i}:=\Omega\backslash(\mathcal{B}_{i}\cup \mathcal{C}_{i});$
\item[$(ii)$]  In $\mathcal{B}_{i}$, $$c_{1}|\nabla \varphi_{i,1}|^{2}\leq|\varphi_{i,1}|\leq c_{2}|\nabla \varphi_{i,1}|^{2};$$
\item[$(iii)$] In $\mathcal{C}_{i}$, $$c_{1}|\nabla\varphi_{i,1}|^{2}\leq|\varphi_{i,1}|;$$
\item[$(iv)$] There is $1\leq j\leq d$ with $j\neq i$ so that
$$\varphi_{i,1}-\varphi_{j,1}\leq-c_{3}\ in \ \mathcal{C}_{i};$$
\item[$(v)$]  $$c_{4}|\nabla\varphi_{i,2}|^{2}\leq|\varphi_{i,2}| \ in \ \Omega\ and \ |\varphi_{i,2}|\leq c_{5}|\nabla \varphi_{i,2}|^{2}\ in \  \vartheta_i;$$
\item[$(vi)$] $$\varphi_{i,2}-\varphi_{i,1}\leq -c_{6}\ in\ \Omega\setminus \vartheta_i.$$
\end{description}
\end{Proposition}

\section{Proof of Theorem~\ref{jiudu4}}\label{kaodu3}

We start this section by introducing some notations. For each $1\leq i\leq d$, we set
\begin{equation*}
\left\{
\begin{array}{l}
\varphi_{i}(x):=\varphi_{i,1}(x),\;x\in \Omega,\\
\\
\varphi_{d+i}(x):=\varphi_{i,2}(x),\;x\in \Omega
\end{array}
\right.
\end{equation*}
and
\begin{equation}\label{12-22-add-2}
\left\{
\begin{array}{l}
\Phi_{i}(x,t):=\displaystyle{\frac{s}{\Gamma(t)}}\varphi_{i,1}(x),\;(x,t)\in \Omega\times [0,T],\\
\\
\Phi_{d+i}(x,t):=\displaystyle{\frac{s}{\Gamma(t)}}\varphi_{i,2}(x),\;(x,t)\in \Omega\times [0,T],
\end{array}
\right.
\end{equation}
with $s\in (0,1]$, $\Gamma(t)=T-t+h$ and $h\in(0,1].$  For each $1\leq i\leq 2d$, we define
$$
f_{i}:=ue^{\Phi_{i}/2}\;\;\mbox{and}\;\;\eta_i:=\frac{1}{2}\partial_t \Phi_i+\frac{1}{4}A\nabla \Phi_i\cdot\nabla\Phi_i,\;\;(x,t)\in \Omega\times (0,T),
$$
where $u$ is the solution to \eqref{1.1}.  By \eqref{1.1}, we find that
\begin{eqnarray}\label{absd7}
&&\partial_{t}f_{i}-\mathrm{div}(A\nabla f_{i})\nonumber\\
&=& -af_{i}+\displaystyle{\frac{1}{2}}f_{i}B\cdot\nabla\Phi_{i}+\eta_{i}f_{i}-
\displaystyle{\frac{1}{2}}\mathrm{div}(A\nabla\Phi_{i})f_{i}-A\nabla\Phi_{i}\cdot\nabla f_{i}-B\cdot\nabla f_{i}\\
&=& -af_{i}+\displaystyle{\frac{1}{2}}f_{i}B\cdot\nabla\Phi_{i}+\eta_{i}f_{i}+\mathcal{A}_{\varphi_{i}}f_{i}-B\cdot\nabla f_{i}\;\;\;
 \mathrm{in}\;\; \Omega\times (0,T),\nonumber
\end{eqnarray}
and
\begin{eqnarray}\label{p20}
A\nabla f_{i}\cdot \overrightarrow{n}-\frac{1}{2}f_{i}A\nabla\Phi_{i}\cdot \overrightarrow{n}=0\;\;\; \mathrm{on} \;\; \partial\Omega\times(0,T).
\end{eqnarray}

Let $\mathbf{f}:=(f_i)_{1\leq i\leq 2d}$, $\mathcal{S}\mathbf{f}:=(\mathcal{S}_{\varphi_{i}}f_{i})_{1\leq i\leq2d}$, $\mathcal{A} \mathbf{f}:=(\mathcal{A}_{\varphi_{i}}f_{i})_{1\leq i\leq2d}$, $\mathbf{F}:=(-af_{i}+\frac{1}{2}f_{i}B\cdot\nabla\Phi_{i}-B\cdot\nabla f_{i})_{1\leq i\leq2d}$
and $\mathcal{S}^\prime \mathbf{f}:=(\mathcal{S}^\prime_{\varphi_{i}}f_{i})_{1\leq i\leq2d}$.
Then by (\ref{absd7})and (\ref{absd8}), we have that
\begin{eqnarray}\label{p19}
&&\partial_{t}\mathbf{f}+\mathcal{S}\mathbf{f}=\mathcal{A} \mathbf{f}+\mathbf{F}\;\;\;\mbox{in}\;\;\Omega\times (0,T).
\end{eqnarray}
The proof of Theorem~\ref{jiudu4} will be carried out by the following five stages.\\

\textbf{Stage 1.} We claim that

\begin{eqnarray}\label{p0008}
\left\{
\begin{array}{l}
\displaystyle\langle\mathcal{A} \mathbf{f}, \mathbf{f}\rangle=0,\\
\\
\displaystyle\langle\mathcal{S} \mathbf{f}, \mathbf{f}\rangle=\sum\limits_{i=1}^{2d}\int_{\Omega}A\nabla f_{i}\cdot\nabla f_{i}\mathrm{d}x-\sum\limits_{i=1}^{2d}\int_{\Omega}\eta_{i}f_{i}^2\mathrm{d}x,\\
\\
\displaystyle\frac{\mathrm{d}}{\mathrm{d}t}\langle\mathcal{S} \mathbf{f}, \mathbf{f}\rangle=2\langle\mathcal{S} \mathbf{f}, \mathbf{f}_{t}\rangle+\langle\mathcal{S'} \mathbf{f}, \mathbf{f}\rangle.
\end{array}
\right.
\end{eqnarray}

For this purpose, firstly,  we observe that
\begin{eqnarray*}
\langle\mathcal{A} \mathbf{f}, \mathbf{f}\rangle
&=&-\sum_{i=1}^{2d}\int_{\Omega}A\nabla\Phi_{i}\cdot\nabla f_{i}f_{i}\mathrm{d}x-\frac{1}{2}\sum_{i=1}^{2d}\int_{\Omega}\mathrm{div}(A\nabla\Phi_{i})f_{i}^2\mathrm{d}x\\
&=&-\sum_{i=1}^{2d}\int_{\Omega}A\nabla\Phi_{i}\cdot\nabla f_{i}f_{i}\mathrm{d}x+\sum_{i=1}^{2d}\int_{\Omega}A\nabla\Phi_{i}\cdot\nabla f_{i}f_{i}\mathrm{d}x-\frac{1}{2}\sum_{i=1}^{2d}\int_{\partial\Omega}A\nabla\Phi_{i}\cdot \overrightarrow{n}f_{i}^{2}\mathrm{d}S\\
&=&-\frac{1}{2}\sum_{i=1}^{2d}\int_{\partial\Omega}A\nabla\Phi_{i}\cdot \overrightarrow{n}f_{i}^{2}\mathrm{d}S,
\end{eqnarray*}
which, combined with (\ref{12-22-add-2}) and (\ref{absd6}), indicates that
\begin{eqnarray}
\langle\mathcal{A} \mathbf{f}, \mathbf{f}\rangle
&=&-\frac{1}{2}\left(\sum_{i=1}^{d}\int_{\partial\Omega}A\nabla\Phi_{i}\cdot \overrightarrow{n}f_{i}^{2}\mathrm{d}S+\sum_{i=1}^{d}\int_{\partial\Omega}A\nabla\Phi_{d+i}\cdot \overrightarrow{n}f_{d+i}^{2}\mathrm{d}S\right)\nonumber\\
&=&-\frac{s}{2\Gamma}\left(\sum_{i=1}^{d}\int_{\partial\Omega}A\nabla\varphi_{i,1}\cdot \overrightarrow{n}f_{i}^{2}\mathrm{d}S+\sum_{i=1}^{d}\int_{\partial\Omega}A\nabla\varphi_{i,2}\cdot \overrightarrow{n}f_{i}^{2}\mathrm{d}S\right)\label{12-17-1}\\
&=&0.\nonumber
\end{eqnarray}

Secondly, by (\ref{p20}) and (\ref{12-17-1}), we obtain that
\begin{eqnarray*}
\langle\mathcal{S} \mathbf{f}, \mathbf{f}\rangle&=&\sum_{i=1}^{2d}\int_{\Omega}A\nabla f_{i}\cdot\nabla f_{i}\mathrm{d}x-\sum_{i=1}^{2d}\int_{\Omega}\eta_{i}f_{i}^2\mathrm{d}x-\sum_{i=1}^{2d}\int_{\partial\Omega}A\nabla f_{i}\cdot \overrightarrow{n}f_{i}\mathrm{d}S\\
&=&\sum_{i=1}^{2d}\int_{\Omega}A\nabla f_{i}\cdot\nabla f_{i}\mathrm{d}x-\sum_{i=1}^{2d}\int_{\Omega}\eta_{i}f_{i}^2\mathrm{d}x-\frac{1}{2}\sum_{i=1}^{2d}\int_{\partial\Omega}A\nabla \Phi_{i}\cdot \overrightarrow{n}f_{i}^2\mathrm{d}S\\
&=&\sum_{i=1}^{2d}\int_{\Omega}A\nabla f_{i}\cdot\nabla f_{i}\mathrm{d}x-\sum_{i=1}^{2d}\int_{\Omega}\eta_{i}f_{i}^2\mathrm{d}x.
\end{eqnarray*}

Finally, it follows from the latter that
\begin{equation}\label{12-17-2}
\begin{array}{lll}
\displaystyle{\frac{\mathrm{d}}{\mathrm{d}t}}\langle\mathcal{S} \mathbf{f}, \mathbf{f}\rangle&=&
 2\displaystyle{\sum_{i=1}^{2d}\int_{\Omega}}A\nabla \partial_{t}f_{i}\cdot\nabla f_{i}\mathrm{d}x-\displaystyle{\sum_{i=1}^{2d}\int_{\Omega}}\partial_{t}\eta_{i}f_{i}^2\mathrm{d}x-
 2\displaystyle{\sum_{i=1}^{2d}\int_{\Omega}}\eta_{i}\partial_{t}f_{i}f_{i}\mathrm{d}x\\
&=& -2\displaystyle{\sum_{i=1}^{2d}\int_{\Omega}}\mathrm{div}(A\nabla f_{i}) \partial_{t}f_{i}\mathrm{d}x+2\displaystyle{\sum_{i=1}^{2d}\int_{\partial\Omega}}A\nabla f_{i}\cdot \overrightarrow{n} \partial_{t}f_{i}\mathrm{d}S\\
&&-\displaystyle{\sum_{i=1}^{2d}\int_{\Omega}}\partial_{t}\eta_{i}f_{i}^2\mathrm{d}x
-2\displaystyle{\sum_{i=1}^{2d}\int_{\Omega}}\eta_{i}\partial_{t}f_{i}f_{i}\mathrm{d}x\\
&=&2\langle\mathcal{S} \mathbf{f}, \mathbf{f}_{t}\rangle-\displaystyle{\sum_{i=1}^{2d}\int_{\Omega}}\partial_{t}\eta_{i}f_{i}^2\mathrm{d}x
+2\displaystyle{\sum_{i=1}^{2d}\int_{\partial\Omega}}A\nabla f_{i}\cdot \overrightarrow{n} \partial_{t}f_{i}\mathrm{d}S.
\end{array}
\end{equation}
By (\ref{12-17-2}), (\ref{p20}) and similar arguments as those to get (\ref{12-17-1}), we have that
\begin{eqnarray*}
\frac{\mathrm{d}}{\mathrm{d}t}\langle\mathcal{S} \mathbf{f}, \mathbf{f}\rangle
&=&2\langle\mathcal{S} \mathbf{f}, \mathbf{f}_{t}\rangle-\sum_{i=1}^{2d}\int_{\Omega}\partial_{t}\eta_{i}f_{i}^2\mathrm{d}x+\sum_{i=1}^{2d}\int_{\partial\Omega}A\nabla\Phi_{i}\cdot \overrightarrow{n}f_{i}\partial_{t}f_{i}\mathrm{d}S\\
&=&2\langle\mathcal{S} \mathbf{f}, \mathbf{f}_{t}\rangle+\langle\mathcal{S'} \mathbf{f}, \mathbf{f}\rangle.
\end{eqnarray*}

\textbf{Stage 2.} We show the following two inequalities:
\begin{equation}\label{p0009}
\left\{
\begin{array}{l}
\displaystyle\left|\frac{1}{2}\frac{\mathrm{d}}{\mathrm{d}t}\|\mathbf{f}\|^2+\mathbf{N}(t)\|\mathbf{f}\|^2\right|
\leq \left(\|a\|_{\infty}+\max_{1\leq i\leq2d}\|\nabla\Phi_{i}\|\|B\|_{\infty}\right)\|\mathbf{f}\|^{2}+\|B\|_{\infty}\|\nabla\mathbf{f}\|\|\mathbf{f}\|,\\
\\
\displaystyle\mathbf{N'}(t)\leq \frac{2\langle\mathcal{S} \mathbf{f}, \mathcal{A} \mathbf{f}\rangle+\langle\mathcal{S'} \mathbf{f}, \mathbf{f}\rangle}{\|\mathbf{f}\|^2}+2\left(\|a\|_{\infty}+\max_{1\leq i\leq2d}\|\nabla\Phi_{i}\|\|B\|_{\infty}\right)^{2}+2\|B\|_{\infty}^{2}\frac{\|\nabla\mathbf{f}\|^{2}}{\|\mathbf{f}\|^{2}},
\end{array}
\right.
\end{equation}
where $\mathbf{N}(t)=\langle\mathcal{S} \mathbf{f}, \mathbf{f}\rangle/\|\mathbf{f}\|^2$.

Indeed, on one hand, according to (\ref{p19}) and the first equality in (\ref{p0008}), we get that
\begin{equation}\label{sdfq1}
\frac{1}{2}\frac{\mathrm{d}}{\mathrm{d}t}\|\mathbf{f}\|^2+\langle\mathcal{S} \mathbf{f}, \mathbf{f}\rangle=\langle\mathbf{F}, \mathbf{f}\rangle.
\end{equation}
On the other hand, by (\ref{p19}) and (\ref{p0008}), we have that
\begin{eqnarray}\label{sdfq2}
\displaystyle\mathbf{N'}(t)\|\mathbf{f}\|^4&=&(2\langle\mathcal{S} \mathbf{f}, \mathbf{f}_{t}\rangle+\langle\mathcal{S'} \mathbf{f}, \mathbf{f}\rangle)\|\mathbf{f}\|^2-\langle\mathcal{S} \mathbf{f}, \mathbf{f}\rangle(-2\langle\mathcal{S} \mathbf{f}, \mathbf{f}\rangle+2\langle\mathbf{F}, \mathbf{f}\rangle)\nonumber\\
&=&(2\langle\mathcal{S} \mathbf{f}, \mathcal{A} \mathbf{f}\rangle+\langle\mathcal{S'} \mathbf{f}, \mathbf{f}\rangle)\|\mathbf{f}\|^2-2\|\mathcal{S}f\|^{2}\|\mathbf{f}\|^2+2\langle\mathcal{S} \mathbf{f}, \mathbf{F}\rangle\|\mathbf{f}\|^2\nonumber\\
&&+2\langle\mathcal{S} \mathbf{f}, \mathbf{f}\rangle^{2}-2\langle\mathcal{S} \mathbf{f}, \mathbf{f}\rangle\langle\mathbf{F}, \mathbf{f}\rangle\\
&=&(2\langle\mathcal{S} \mathbf{f}, \mathcal{A} \mathbf{f}\rangle+\langle\mathcal{S'} \mathbf{f}, \mathbf{f}\rangle)\|\mathbf{f}\|^2-2\|\mathcal{S}f-\frac{1}{2}\mathbf{F}\|^{2}\|\mathbf{f}\|^2\nonumber\\
&&+\displaystyle{\frac{1}{2}}\|\mathbf{F}\|^{2} \|\mathbf{f}\|^2+2\langle\mathcal{S} \mathbf{f}-\frac{1}{2}\mathbf{F}, \mathbf{f}\rangle^{2}-\frac{1}{2}\langle\mathbf{F}, \mathbf{f}\rangle^{2}\nonumber\\
&\leq&(2\langle\mathcal{S} \mathbf{f}, \mathcal{A} \mathbf{f}\rangle+\langle\mathcal{S'} \mathbf{f}, \mathbf{f}\rangle)\|\mathbf{f}\|^2+\|\mathbf{F}\|^{2}\|\mathbf{f}\|^2.\nonumber
\end{eqnarray}
Note that
\begin{eqnarray*}\label{sdfq3}
\|\mathbf{F}\|&=&\|(-af_{i}+\frac{1}{2}f_{i}B\cdot\nabla\Phi_{i}-B\cdot\nabla f_{i})_{1\leq i\leq2d}\|\nonumber\\
&\leq&\left(\|a\|_{\infty}+\max_{1\leq i\leq2d}\|\nabla\Phi_{i}\|\|B\|_{\infty}\right)\|\mathbf{f}\|+\|B\|_{\infty}\|\nabla\mathbf{f}\|.
\end{eqnarray*}
This, along with (\ref{sdfq1}) and (\ref{sdfq2}), yields (\ref{p0009}).\\

\textbf{Stage 3.} We claim that for $s:=s(A,\Omega,\widetilde{\omega})\in (0,1]$ sufficiently small,
\begin{equation}\label{12-17-3}
\eta_{i}\leq0,\;\;\;\langle\mathcal{S} \mathbf{f}, \mathbf{f}\rangle\geq0,
\end{equation}
and
\begin{eqnarray}\label{sdfqq3}
\langle\mathcal{S'} \mathbf{f}, \mathbf{f}\rangle+2\langle\mathcal{S} \mathbf{f}, \mathcal{A}\mathbf{f}\rangle\leq\frac{1+C_{0}(A,\Omega,\widetilde{\omega})}{\Gamma}\langle\mathcal{S} \mathbf{f}, \mathbf{f}\rangle+\frac{C(A,\Omega,\widetilde{\omega})}{h^{2}}\|\mathbf{f}\|^{2},
\end{eqnarray}
where $C_{0}(A,\Omega,\widetilde{\omega})\in (0,1)$ and $C(A,\Omega,\widetilde{\omega})>0$.

To this end, by (\ref{12-22-add-2}) and (\ref{12-22-add}), we firstly observe that
\begin{equation}\label{12-17-4}
	\displaystyle\eta_{i}:=\frac{1}{2}\partial_{t}\Phi_{i}+\frac{1}{4}A\nabla\Phi_{i}\cdot\nabla\Phi_{i}=
\begin{cases}
	\displaystyle\frac{s}{4\Gamma^{2}}\left(-2|\varphi_{i,1}|+sA\nabla\varphi_{i,1}\cdot\nabla\varphi_{i,1}\right),& 1\leq i\leq d,\\
\\ \displaystyle\frac{s}{4\Gamma^{2}}\left(-2|\varphi_{i-d,2}|+sA\nabla\varphi_{i-d,2}\cdot\nabla\varphi_{i-d,2}\right),& d+1\leq i\leq 2d.
\end{cases}
\end{equation}
By $(i), (ii), (iii)$ and $(v)$ in Proposition~\ref{lemma-1}, we have that
$$
|\nabla\varphi_{i,1}|^{2}\leq \frac{1}{c_{1}}|\varphi_{i,1}|\;\; \mbox{and}\;\;
|\nabla\varphi_{i,2}|^{2}\leq \frac{1}{c_{4}}|\varphi_{i,2}|
$$
for each $ 1\leq i\leq d$. Thus, for $s:=s(A,\Omega,\widetilde{\omega})\in(0,1]$ sufficiently small, we obtain that
\begin{equation}\label{12-17-5}
\begin{array}{lll}
&&-2|\varphi_{i,1}|+sA\nabla\varphi_{i,1}\cdot\nabla\varphi_{i,1}\\
&\leq& -2|\varphi_{i,1}|+s\lambda|\nabla\varphi_{i,1}|^{2}\leq 0,\;\;\;\;\;\forall\ 1\leq i\leq d
\end{array}
\end{equation}
and
\begin{equation}\label{12-17-6}
\begin{array}{lll}
&&-2|\varphi_{i-d,2}|+sA\nabla\varphi_{i-d,2}\cdot\nabla\varphi_{i-d,2}\\
&\leq& -2|\varphi_{i-d,2}|+s\lambda|\nabla\varphi_{i-d,2}|^{2}\leq 0,\;\;\;\;\;\forall\;d+1\leq i\leq 2d.
\end{array}
\end{equation}
Then (\ref{12-17-3}) follows from (\ref{12-17-4})-(\ref{12-17-6}) and (\ref{p0008}).\\

Next, we turn to the proof of (\ref{sdfqq3}). According to $(iii)$ in Proposition~\ref{lemma-2}, it is clear that
\begin{eqnarray}
&&\langle\mathcal{S'} \mathbf{f}, \mathbf{f}\rangle+2\langle\mathcal{S} \mathbf{f}, \mathcal{A}\mathbf{f}\rangle\nonumber\\
&=&(I_1)+(I_2)-\displaystyle{\frac{2}{\Gamma}\sum_{m=1}^{2d}\int_{\Omega}} \left(\eta_{m}+\frac{1}{4}A\nabla\Phi_{m}\cdot\nabla\Phi_{m}+\frac{s}{8}A_{ij}\partial_{x_{j}}\varphi_{m}\partial_{x_{i}}A_{kl}\partial_{x_{l}}\Phi_{m}\partial_{x_{k}}\Phi_{m}\right)  f_{m}^2 \mathrm{d}x\nonumber\\
&&-\displaystyle{\frac{2}{\Gamma}\sum_{m=1}^{2d}\int_{\Omega}} \left(\frac{s}{4}A\nabla^{2}\varphi_{m} A\nabla\Phi_{m}\cdot\nabla\Phi_{m}\right)  f_{m}^2 \mathrm{d}x,\label{p888}
\end{eqnarray}
where
\begin{equation*}
\begin{array}{lll}
(I_1) :&=&-2\displaystyle{\sum_{m=1}^{2d}\int_{\Omega}} A_{ij}\partial_{x_{j}}f_{m}\partial_{x_{i}}A_{kl}\partial_{x_{l}}f_{m}\partial_{x_{k}}\Phi_{m} \mathrm{d}x+\displaystyle{\sum_{m=1}^{2d}\int_{\Omega}}\partial_{x_{l}}A_{ij}\partial_{x_{j}}f_{m}A_{kl}\partial_{x_{i}}f_{m}\partial_{x_{k}}\Phi_{m} \mathrm{d}x\\
&&-2\displaystyle{\sum_{m=1}^{2d}\int_{\Omega}} A\nabla^{2} \Phi_{m} A\nabla f_{m}\cdot \nabla f_{m} \mathrm{d}x-\displaystyle{\sum_{m=1}^{2d}\int_{\Omega}} A\nabla f_{m}\cdot\nabla\left[\mathrm{div}(A\nabla \Phi_{m})\right] f_{m}\mathrm{d}x\\
\end{array}
\end{equation*}
and
\begin{equation*}
\begin{array}{lll}
(I_2) :&=& 2\displaystyle{\sum_{m=1}^{2d}\int_{\partial\Omega}}(A\nabla f_{m}\cdot \overrightarrow{n})( A\nabla\Phi_{m}\cdot\nabla f_{m})\mathrm{d}S-\displaystyle{\sum_{m=1}^{2d}\int_{\partial\Omega}}(A\nabla \Phi_{m} \cdot \overrightarrow{n}) (A\nabla f_{m}\cdot\nabla f_{m})\mathrm{d}S\\
&&+\displaystyle{\sum_{m=1}^{2d}\int_{\partial\Omega}}(A\nabla f_{m}\cdot \overrightarrow{n})\mathrm{div}(A\nabla \Phi_{m}) f_{m}\mathrm{d}S+\displaystyle{\sum_{m=1}^{2d}\int_{\partial\Omega}}(A\nabla \Phi_{m} \cdot \overrightarrow{n})\eta_{m} f^{2}_{m}\mathrm{d}S.
\end{array}
\end{equation*}
We will estimate the right hand sum in (\ref{p888}). This will be done by four steps as follows.\\

\textbf{Step 1.} We show that
\begin{equation}\label{12-17-9}
(I_1)\leq \frac{sC(A,\Omega,\widetilde{\omega})}{\Gamma}\sum_{m=1}^{2d}\int_\Omega |\nabla f_m|^2\mathrm{d}x
+\frac{sC(A,\Omega,\widetilde{\omega})}{h}\int_\Omega |\mathbf{f}|^2\mathrm{d}x.
\end{equation}
Indeed, after some simple calculations, we can directly check that
\begin{eqnarray*}
(I_1)&\leq& \displaystyle{\frac{sC(A,\Omega,\widetilde{\omega})}{\Gamma}\sum_{m=1}^{2d}\int_\Omega} |\nabla f_m|^2\mathrm{d}x
+\displaystyle{\frac{sC(A,\Omega,\widetilde{\omega})}{\Gamma}\sum_{m=1}^{2d}\int_\Omega} |\nabla f_m| |f_m|\mathrm{d}x\\
&\leq& \displaystyle{\frac{sC(A,\Omega,\widetilde{\omega})}{\Gamma}\sum_{m=1}^{2d}\int_\Omega} |\nabla f_m|^2\mathrm{d}x
+\displaystyle{\frac{sC(A,\Omega,\widetilde{\omega})}{h}\int_\Omega} |\mathbf{f}|^2\mathrm{d}x,
\end{eqnarray*}
which indicates (\ref{12-17-9}).\\

\textbf{Step 2.} We claim that
\begin{equation}\label{12-17-10}
\begin{array}{lll}
(I_2)&\leq& \displaystyle{\frac{s^2 C(A,\Omega,\widetilde{\omega})}{h^2}\sum_{m=1}^{d}\int_\Omega} |f_m|^2\mathrm{d}x\\
&&+\displaystyle{\frac{sC(A,\Omega,\widetilde{\omega})}{\Gamma}\sum_{m=1}^{d}\int_\Omega} |\nabla \Phi_m|^2 |f_m|^2\mathrm{d}x
+\displaystyle{\frac{sC(A,\Omega,\widetilde{\omega})}{\Gamma}\sum_{m=1}^{d}\int_\Omega} |\nabla f_m|^2\mathrm{d}x.
\end{array}
\end{equation}

Since $\psi_i=0$ on $\partial\Omega$ ($1\leq i\leq d$), it holds that $\nabla \psi_i=\partial_{\vec{n}} \psi_i \vec{n}$ on
$\partial\Omega\times (0,T)$. This yields that
\begin{equation}\label{12-17-11}
\nabla \Phi_m=\partial_{\vec{n}} \Phi_m \vec{n}\;\;\mbox{on}\;\;\partial\Omega\times (0,T),\;\;\;\;\forall\;1\leq m\leq 2d.
\end{equation}
Firstly, by (\ref{p20}), (\ref{12-17-11}), (\ref{12-22-add-2}) and (\ref{absd6}), we have that
\begin{equation}\label{12-17-12}
\begin{array}{lll}
&&2\displaystyle{\sum_{m=1}^{2d}\int_{\partial\Omega}}(A\nabla f_{m}\cdot \overrightarrow{n})( A\nabla\Phi_{m}\cdot\nabla f_{m})\mathrm{d}S\\
&=&2\displaystyle{\sum_{m=1}^{d}\int_{\partial\Omega}}\left(\frac{1}{2}A\nabla \Phi_{m}\cdot \overrightarrow{n}f_{m}\right)^{2}\partial_{\vec{n}}\Phi_{m}\mathrm{d}S
+2\displaystyle{\sum_{m=1}^{d}\int_{\partial\Omega}}\left(\frac{1}{2}A\nabla \Phi_{d+m}\cdot \overrightarrow{n}f_{d+m}\right)^{2}\partial_{\vec{n}}\Phi_{d+m}\mathrm{d}S\\
&=&2\displaystyle{\sum_{m=1}^{d}\int_{\partial\Omega}}\left(\frac{1}{2}A\nabla \Phi_{m}\cdot \overrightarrow{n}f_{m}\right)^{2}\partial_{\vec{n}}\Phi_{m}\mathrm{d}S
-2\displaystyle{\sum_{m=1}^{d}\int_{\partial\Omega}}\left(\frac{1}{2}A\nabla \Phi_{m}\cdot \overrightarrow{n}f_{m}\right)^{2}\partial_{\vec{n}}\Phi_{m}\mathrm{d}S\\
&=&0.
\end{array}
\end{equation}
Secondly, it follows from  (\ref{absd6}), (\ref{12-22-add-2}),  (\ref{12-17-11}) and (\ref{1.1})  that
\begin{equation}\label{12-17-13}
\begin{array}{lll}
&&-\displaystyle{\sum_{m=1}^{2d}\int_{\partial\Omega}}(A\nabla \Phi_{m} \cdot \overrightarrow{n}) (A\nabla f_{m}\cdot\nabla f_{m})\mathrm{d}S\\
&=&-\displaystyle{\sum_{m=1}^{d}\int_{\partial\Omega}}(A\nabla \Phi_{m} \cdot \overrightarrow{n}) (A\nabla f_{m}\cdot\nabla f_{m}-A\nabla f_{d+m}\cdot\nabla f_{d+m})\mathrm{d}S\\
&=&-\displaystyle{\sum_{m=1}^{d}\int_{\partial\Omega}}(A\nabla \Phi_{m} \cdot \overrightarrow{n})e^{\Phi_{m}} \left[A\nabla u\cdot\nabla u+uA\nabla u\cdot\nabla\Phi_{m}+\frac{1}{4}u^{2}A\nabla \Phi_{m}\cdot\nabla \Phi_{m}\right]\mathrm{d}S\\
&&+\displaystyle{\sum_{m=1}^{d}\int_{\partial\Omega}}(A\nabla \Phi_{m} \cdot \overrightarrow{n}) e^{\Phi_{m}}\left[A\nabla u\cdot\nabla u-uA\nabla u\cdot\nabla\Phi_{m}+\frac{1}{4}u^{2}A\nabla \Phi_{m}\cdot\nabla \Phi_{m}\right]\mathrm{d}S\\
&=&-2\displaystyle{\sum_{m=1}^{d}\int_{\partial\Omega}}(A\nabla \Phi_{m} \cdot \overrightarrow{n})e^{\Phi_{m}} uA\nabla u\cdot \overrightarrow{n}\partial_{\vec{n}}\Phi_{m}\mathrm{d}S\\
&=&0.
\end{array}
\end{equation}
Thirdly, by (\ref{12-22-add-2}) and (\ref{absd6}), we have that
\begin{equation}\label{12-17-14}
\begin{array}{lll}
&&\displaystyle{\sum_{m=1}^{2d}\int_{\partial\Omega}}(A\nabla \Phi_{m} \cdot \overrightarrow{n})\eta_{m} f^{2}_{m}\mathrm{d}S\\
&=&\displaystyle{\sum_{m=1}^{d}\int_{\partial\Omega}}\left(\frac{1}{2}\partial_{t}\Phi_{m}+\frac{1}{4}A\nabla\Phi_{m}\cdot\nabla\Phi_{m}\right)(A\nabla \Phi_{m} \cdot \overrightarrow{n})u^{2}e^{\Phi_{m}}\mathrm{d}S\\
&&+\displaystyle{\sum_{m=1}^{d}\int_{\partial\Omega}}\left(\frac{1}{2}\partial_{t}\Phi_{m+d}+\frac{1}{4}A\nabla\Phi_{m+d}\cdot\nabla\Phi_{m+d}\right)(A\nabla \Phi_{m+d} \cdot \overrightarrow{n})u^{2}e^{\Phi_{m+d}}\mathrm{d}S\\
&=&\displaystyle{\sum_{m=1}^{d}\int_{\partial\Omega}}\left(\frac{1}{2}\partial_{t}\Phi_{m}+\frac{1}{4}A\nabla\Phi_{m}\cdot\nabla\Phi_{m}\right)(A\nabla \Phi_{m} \cdot \overrightarrow{n})u^{2}e^{\Phi_{m}}\mathrm{d}S\\
&&+\displaystyle{\sum_{m=1}^{d}\int_{\partial\Omega}}\left(\frac{1}{2}\partial_{t}\Phi_{m}+\frac{1}{4}A\nabla\Phi_{m}\cdot\nabla\Phi_{m}\right)(A\nabla \Phi_{m+d} \cdot \overrightarrow{n})u^{2}e^{\Phi_{m}}\mathrm{d}S\\
&=&\displaystyle{\sum_{m=1}^{d}\int_{\partial\Omega}}\left(\frac{1}{2}\partial_{t}\Phi_{m}+\frac{1}{4}A\nabla\Phi_{m}\cdot\nabla\Phi_{m}\right)[A(\nabla \Phi_{m}+\nabla \Phi_{m+d}) \cdot \overrightarrow{n}]u^{2}e^{\Phi_{m}}\mathrm{d}S\\
&=&0.
\end{array}
\end{equation}
Finally, it follows from (\ref{12-22-add-2}), (\ref{1.1}) and (\ref{absd6}) that
\begin{eqnarray*}
&&\sum_{m=1}^{2d}\int_{\partial\Omega}(A\nabla f_{m}\cdot \overrightarrow{n})\mathrm{div}(A\nabla \Phi_{m}) f_{m}\mathrm{d}S\\
&=&\sum_{m=1}^{d}\int_{\partial\Omega}A\nabla \Phi_{m}\cdot \overrightarrow{n}\mathrm{div}(A\nabla \Phi_{m}) f^{2}_{m}\mathrm{d}S\\
&=&-\sum_{m=1}^{d}\int_{\Omega}|\mathrm{div}(A\nabla \Phi_{m})|^{2} f^{2}_{m}\mathrm{d}x+\sum_{m=1}^{d}\int_{\Omega}A\nabla \Phi_{m}\cdot \nabla\left[\mathrm{div}(A\nabla \Phi_{m}) f^{2}_{m}\right]\mathrm{d}x\\
&\leq&\frac{s^2C(A,\Omega,\widetilde{\omega})}{h^2}\sum_{m=1}^{d}\int_{\Omega}f^{2}_{m}\mathrm{d}x
+\frac{sC(A,\Omega,\widetilde{\omega})}{\Gamma}\sum_{m=1}^{d}\int_{\Omega}|\nabla \Phi_{m}|^2f^{2}_{m}\mathrm{d}x+\frac{sC(A,\Omega,\widetilde{\omega})}{\Gamma}\sum_{m=1}^{d}\int_{\Omega}|\nabla f_{m}|^{2}\mathrm{d}x.
\end{eqnarray*}
This, along with (\ref{12-17-12})-(\ref{12-17-14}), yields (\ref{12-17-10}).\\

\textbf{Step 3.} We show that for $s:=s(A,\Omega,\widetilde{\omega})\in (0,1]$ sufficiently small,
\begin{eqnarray}\label{p88804}
-\frac{2}{\Gamma}\sum_{m=1}^{2d}\int_{\Omega} \left(\eta_{m}+\frac{1}{8\lambda}|\nabla \Phi_{m}|^2\right)  f_{m}^2 \mathrm{d}x
\leq\frac{C(A,\Omega,\widetilde{\omega})}{h^{2}}\|\mathbf{f}\|^{2}+\frac{2-s/(c\lambda)}{\Gamma}\sum_{m=1}^{2d}\int_{\Omega} \left(-\eta_{m}\right)  f_{m}^2 \mathrm{d}x,
\end{eqnarray}
where $c:=2(c_{2}+c_{5}), c_{2}$ and $c_{5}$ are given by Proposition~\ref{lemma-1}.

Firstly, by $(i)$ and $(ii)$ in Proposition~\ref{lemma-1}, we have that $|\varphi_{m,1}|\leq c_{2}|\nabla \varphi_{m,1}|^{2}$ in $\mathcal{B}_{m}\cup \mathcal{D}_{m}$ for each $1\leq m\leq d$. This implies that
\begin{eqnarray*}
-|\nabla \Phi_{m}|^2&=&-\displaystyle{\frac{s^{2}}{\Gamma^{2}}}|\nabla \varphi_{m,1}|^{2}\leq-\displaystyle{\frac{s^{2}}{c_{2}\Gamma^{2}}}
|\varphi_{m,1}|\\
&=&\displaystyle{\frac{2s}{c_{2}}}\left(-\displaystyle{\frac{s}{2\Gamma^{2}}}|\varphi_{m,1}|\right)\leq \displaystyle{\frac{2s}{c_{2}}}\eta_{m}
\;\;\;\mbox{in}\;\;\mathcal{B}_{m}\cup \mathcal{D}_{m},\;\;\forall\;1\leq m\leq d.
\end{eqnarray*}
Here, we used (\ref{12-22-add-2}) and (\ref{12-22-add}). The latter estimate yields that
\begin{equation}\label{12-17-15}
-\frac{2}{\Gamma}\sum_{m=1}^{d}\int_{\mathcal{B}_{m}\cup \mathcal{D}_{m}} \left(\eta_{m}+\frac{1}{8\lambda}|\nabla \Phi_{m}|^2\right)  f_{m}^2 \mathrm{d}x\leq \frac{2-s/(2c_{2}\lambda)}{\Gamma}\sum_{m=1}^{d}\int_{\mathcal{B}_{m}\cup \mathcal{D}_{m}} \left(-\eta_{m}\right)  f_{m}^2 \mathrm{d}x.
\end{equation}
Secondly, according to $(iv)$ in  Proposition~\ref{lemma-1}, there is $c_{3}>0$ so that
for each $1\leq m\leq d$, there exists $1\leq p_m\leq d$ with $p_m\not=m$ satisfying
\begin{equation}\label{12-17-16}
\varphi_{m,1}-\varphi_{p_m,1}\leq-c_{3}\;\;\mbox{in}\;\;\mathcal{C}_{m}.
\end{equation}
Note that
\begin{equation}\label{12-17-17}
\left|\eta_{m}+\frac{1}{8\lambda}|\nabla \Phi_{m}|^2\right|\leq \frac{sC(A,\Omega,\widetilde{\omega})}{\Gamma^{2}}\;\;
\mbox{and}\;\;f_{m}^{2}=e^{s(\varphi_{m,1}-\varphi_{p_m,1})\frac{1}{\Gamma}}f_{p_m}^{2},\;\;\;\forall\;1\leq m\leq d.
\end{equation}
It follows from (\ref{12-17-16}) and (\ref{12-17-17}) that
\begin{equation}\label{12-17-18}
\begin{array}{lll}
-\displaystyle{\frac{2}{\Gamma}\sum_{m=1}^{d}\int_{\mathcal{C}_{m}}} \left(\eta_{m}+\frac{1}{8\lambda}|\nabla \Phi_{m}|^2\right)  f_{m}^2 \mathrm{d}x
&\leq& \displaystyle{\frac{sC(A,\Omega,\widetilde{\omega})}{\Gamma^{3}}\sum_{m=1}^{d}\int_{\mathcal{C}_{m}}} f_{m}^2 \mathrm{d}x\\
&\leq& \displaystyle{\frac{sC(A,\Omega,\widetilde{\omega})}{\Gamma^{3}}}e^{-c_{3}\frac{s}{\Gamma}}\sum_{m=1}^{d}\int_{\mathcal{C}_{m}} f_{p_m}^2 \mathrm{d}x\\
&\leq& \displaystyle{\frac{sC(A,\Omega,\widetilde{\omega})}{\Gamma^{3}}}e^{-c_{3}\frac{s}{\Gamma}}\|\mathbf{f}\|^{2}\\
&\leq& \displaystyle{\frac{C(A,\Omega,\widetilde{\omega})}{h^{2}}}\|\mathbf{f}\|^{2}.
\end{array}
\end{equation}
Thirdly, by $(v)$ in Proposition~\ref{lemma-1}, we get that
\begin{equation*}
|\varphi_{m,2}|\leq c_{5}|\nabla \varphi_{m,2}|^{2}\;\;\;\mbox{in}\;\;\vartheta_m,\;\;\forall\;1\leq m\leq d.
\end{equation*}
This implies that
\begin{eqnarray*}
&&-|\nabla \Phi_{m}|^2=-|\nabla \Phi_{m-d+d}|^2=-\displaystyle{\frac{s^{2}}{\Gamma^{2}}}|\nabla \varphi_{m-d,2}|^{2}\\
&\leq&-\displaystyle{\frac{s^{2}}{c_{5}\Gamma^{2}}}| \varphi_{m-d,2}|=\displaystyle{\frac{2s}{c_{5}}}\left(-\displaystyle{\frac{s}{2\Gamma^{2}}}
|\varphi_{m-d,2}|\right)\leq\frac{2s}{c_{5}}\eta_{m}\;\;\;\mbox{in}\;\;\vartheta_{m-d},\;\;\forall\;d+1\leq m\leq 2d,
\end{eqnarray*}
which indicates
\begin{equation}\label{12-17-21}
-\frac{2}{\Gamma}\sum_{m=d+1}^{2d}\int_{\vartheta_{m-d}} \left(\eta_{m}+\frac{1}{8\lambda}|\nabla \Phi_{m}|^2\right)  f_{m}^2 \mathrm{d}x\leq \frac{2-s/(2c_{5}\lambda)}{\Gamma}\sum_{m=d+1}^{2d}\int_{\vartheta_{m-d}} \left(-\eta_{m}\right)  f_{m}^2 \mathrm{d}x.
\end{equation}
Finally, according to $(vi)$ in Proposition~\ref{lemma-1}, there is $c_{6}>0$ so that
for each $1\leq m\leq d$,
\begin{equation}\label{12-17-22}
\varphi_{m,2}-\varphi_{m,1}\leq-c_{6}\;\;\;\mbox{in}\;\;\Omega\setminus\vartheta_m.
\end{equation}
Observe that
\begin{equation}\label{12-17-23}
\left|\eta_{d+m}+\frac{1}{8\lambda}|\nabla \Phi_{d+m}|^2\right|\leq \frac{sC(A,\Omega,\widetilde{\omega})}{\Gamma^{2}}
\;\;\mbox{and}\;\;f_{d+m}^{2}=e^{s(\varphi_{m,2}-\varphi_{m,1})\frac{1}{\Gamma}}f_{m}^{2},\;\;\;\forall\;1\leq m\leq d.
\end{equation}
It follows from (\ref{12-17-22}) and (\ref{12-17-23}) that
\begin{eqnarray*}
&&-\frac{2}{\Gamma}\sum_{m=d+1}^{2d}\int_{\Omega\setminus \vartheta_{m-d}} \left(\eta_{m}+\frac{1}{8\lambda}|\nabla \Phi_{m}|^2\right)  f_{m}^2 \mathrm{d}x\\
&\leq& \frac{sC(A,\Omega,\widetilde{\omega})}{\Gamma^{3}}e^{-c_{6}\frac{s}{\Gamma}}\sum_{m=1}^{d}\int_{\Omega\setminus \vartheta_m} f_{m}^2 \mathrm{d}x\\
&\leq& \frac{sC(A,\Omega,\widetilde{\omega})}{\Gamma^{3}}e^{-c_{6}\frac{s}{\Gamma}}\|\mathbf{f}\|^{2}\\
&\leq& \frac{C(A,\Omega,\widetilde{\omega})}{h^{2}}\|\mathbf{f}\|^{2}.
\end{eqnarray*}
This, along with (\ref{12-17-21}), (\ref{12-17-15}), (\ref{12-17-18}) and the first inequality in (\ref{12-17-3}), implies (\ref{p88804}).\\

\textbf{Step 4.} We end the proof of (\ref{sdfqq3}).

It follows from (\ref{p888}), (\ref{12-17-9}) and (\ref{12-17-10}) that
\begin{equation}\label{12-17-24}
\begin{array}{lll}
&&\langle\mathcal{S'} \mathbf{f}, \mathbf{f}\rangle+2\langle\mathcal{S} \mathbf{f}, \mathcal{A}\mathbf{f}\rangle\\
&\leq&\displaystyle{\frac{sC(A,\Omega,\widetilde{\omega})}{\Gamma}\sum_{m=1}^{d}\int_{\Omega}}|\nabla \Phi_{m}|^2f^{2}_{m}\mathrm{d}x
+\displaystyle{\frac{s C(A,\Omega,\widetilde{\omega})}{\Gamma}\sum_{m=1}^{2d}\int_{\Omega}}|\nabla f_{m}|^{2}\mathrm{d}x+
\displaystyle{\frac{s C(A,\Omega,\widetilde{\omega})}{h^{2}}}\|\mathbf{f}\|^{2}\\
&&-\displaystyle{\frac{2}{\Gamma}\sum_{m=1}^{2d}\int_{\Omega}} \left(\eta_{m}+\displaystyle{\frac{1}{4}}A\nabla\Phi_{m}\cdot\nabla\Phi_{m}+
\displaystyle{\frac{s}{8}}A_{ij}\partial_{x_{j}}\varphi_{m}\partial_{x_{i}}A_{kl}\partial_{x_{l}}\Phi_{m}\partial_{x_{k}}\Phi_{m}\right)  f_{m}^2 \mathrm{d}x\\
&&-\displaystyle{\frac{2}{\Gamma}\sum_{m=1}^{2d}\int_{\Omega}} \left(\displaystyle{\frac{s}{4}}A\nabla^{2}\varphi_{m} A\nabla\Phi_{m}\cdot\nabla\Phi_{m}\right)  f_{m}^2 \mathrm{d}x.
\end{array}
\end{equation}
Note that
\begin{equation*}
\lambda^{-1}\nabla\Phi_{m}\cdot\nabla\Phi_{m}\leq A\nabla\Phi_{m}\cdot\nabla\Phi_{m}\leq\lambda\nabla\Phi_{m}\cdot\nabla\Phi_{m},
\end{equation*}
\begin{equation*}
-\frac{s}{4}A\nabla^{2}\varphi_{m} A\nabla\Phi_{m}\cdot\nabla\Phi_{m}\leq sC(A,\Omega,\widetilde{\omega})|\nabla\Phi_{m}|^{2}
\end{equation*}
and
\begin{equation*}
-\frac{s}{8}A_{ij}\partial_{x_{j}}\varphi_{m}\partial_{x_{i}}A_{kl}\partial_{x_{l}}\Phi_{m}\partial_{x_{k}}\Phi_{m}\leq sC(A,\Omega,\widetilde{\omega})|\nabla\Phi_{m}|^{2}.
\end{equation*}
Then by (\ref{12-17-24}), we have that for $s:=s(A,\Omega,\widetilde{\omega})\in (0,1]$ sufficiently small,
\begin{equation}\label{p88803}
\begin{array}{lll}
&&\langle\mathcal{S'} \mathbf{f}, \mathbf{f}\rangle+2\langle\mathcal{S} \mathbf{f}, \mathcal{A}\mathbf{f}\rangle\\
&\leq&\displaystyle{\frac{s C(A,\Omega,\widetilde{\omega})}{\Gamma}\sum_{m=1}^{2d}\int_{\Omega}} |\nabla f_{m}|^{2}\mathrm{d}x+
\displaystyle{\frac{C(A,\Omega,\widetilde{\omega})}{h^{2}}}\|\mathbf{f}\|^{2}-
\displaystyle{\frac{2}{\Gamma}\sum_{m=1}^{2d}\int_{\Omega}} \left(\eta_{m}+\frac{1}{4\lambda}|\nabla \Phi_{m}|^2\right)  f_{m}^2 \mathrm{d}x\\
&&+\displaystyle{\frac{sC(A,\Omega,\widetilde{\omega})}{\Gamma}\sum_{m=1}^{d}\int_{\Omega}}|\nabla \Phi_{m}|^2f^{2}_{m}\mathrm{d}x\\
&\leq&\displaystyle{\frac{s C(A,\Omega,\widetilde{\omega})}{\Gamma}\sum_{m=1}^{2d}\int_{\Omega}}|\nabla f_{m}|^{2}\mathrm{d}x
+\displaystyle{\frac{ C(A,\Omega,\widetilde{\omega})}{h^{2}}}\|\mathbf{f}\|^{2}-\displaystyle{\frac{2}{\Gamma}\sum_{m=1}^{2d}\int_{\Omega}} \left(\eta_{m}+\frac{1}{8\lambda}|\nabla \Phi_{m}|^2\right)  f_{m}^2 \mathrm{d}x.
\end{array}
\end{equation}
It follows from (\ref{p88803}), (\ref{p88804}) and the second equality in (\ref{p0008}) that for  $s:=s(A,\Omega,\widetilde{\omega})\in (0,1]$ sufficiently small,
\begin{eqnarray*}
&&\langle\mathcal{S'} \mathbf{f}, \mathbf{f}\rangle+2\langle\mathcal{S} \mathbf{f}, \mathcal{A}\mathbf{f}\rangle\\
&\leq&\displaystyle{\frac{s C(A,\Omega,\widetilde{\omega})}{\Gamma}\sum_{m=1}^{2d}\int_{\Omega}}A\nabla f_m\cdot \nabla f_m\mathrm{d}x
+\displaystyle{\frac{ C(A,\Omega,\widetilde{\omega})}{h^{2}}}\|\mathbf{f}\|^{2}+\displaystyle{\frac{2-s/(c\lambda)}{\Gamma}\sum_{m=1}^{2d}\int_{\Omega}} (-\eta_{m})f_{m}^2 \mathrm{d}x\\
&\leq &\displaystyle{\frac{1+C_0(A,\Omega,\widetilde{\omega})}{\Gamma}}\langle\mathcal{S} \mathbf{f}, \mathbf{f}\rangle
+\displaystyle{\frac{C(A,\Omega,\widetilde{\omega})}{h^{2}}}\|\mathbf{f}\|^{2},
\end{eqnarray*}
where $C_0(A,\Omega,\widetilde{\omega})\in (0,1)$. Hence, (\ref{sdfqq3}) holds.\\

\textbf{Stage 4.} We show that for $s:=s(A,\Omega,\widetilde{\omega})\in (0,1]$ sufficiently small,
\begin{equation}\label{12-17-25}
\|\mathbf{f}(\cdot,T-lh)\|^{2(1+M)}\leq \|\mathbf{f}(\cdot,T)\|^2
\|\mathbf{f}(\cdot,T-2lh)\|^{2M} e^{(M+1)C(A,\Omega,\widetilde{\omega},l)(1+h^2\|a\|_\infty^2+\|B\|_\infty^2)},\;\;\forall\; M\geq M_l,
\end{equation}
where $l>1$, $0<h\leq \min\{T/(4l),1\}$, and $M_l$ will be precised later (see (\ref{12-17-26}) below).

Indeed, by (\ref{p0009}), the second equality in  (\ref{p0008}), (\ref{yu-11-28-2}), the first inequality in (\ref{12-17-3}) and (\ref{sdfqq3}), we have that
for $s:=s(A,\Omega,\widetilde{\omega})\in (0,1]$ sufficiently small,
\begin{eqnarray}\label{sdfqq4}
&&\left|\frac{1}{2}\frac{d}{dt}\|\mathbf{f}\|^2+\mathbf{N}(t)\|\mathbf{f}\|^2\right|\nonumber\\
&\leq&\left(\|a\|_{\infty}+\max_{1\leq i\leq2d}\|\nabla\Phi_{i}\|\|B\|_{\infty}+\frac{\lambda}{2}\|B\|^{2}_{\infty}\right)\|\mathbf{f}\|^{2}
+\frac{1}{2\lambda}\|\nabla\mathbf{f}\|^{2}\\
&\leq&\left(\|a\|_{\infty}+\max_{1\leq i\leq2d}\|\nabla\Phi_{i}\|\|B\|_{\infty}+\frac{\lambda}{2}\|B\|^{2}_{\infty}\right)\|\mathbf{f}\|^{2}
+\frac{1}{2}\mathbf{N}(t)\|\mathbf{f}\|^2\nonumber
\end{eqnarray}
and
\begin{equation}\label{sdfqq5}
\begin{array}{lll}
\mathbf{N'}(t)&\leq&  \left(2\lambda\|B\|_{\infty}^{2}+\displaystyle{\frac{1+C_{0}(A,\Omega,\widetilde{\omega})}{T-t+h}}\right)\mathbf{N}(t)\\
&&+\displaystyle{\frac{C(A,\Omega,\widetilde{\omega})}{h^{2}}}+2\left(\|a\|_{\infty}+\max_{1\leq i\leq2d}\|\nabla\Phi_{i}\|\|B\|_{\infty}\right)^{2}.
\end{array}
\end{equation}
Let $l>1$ and $0<h\leq \min\{T/(4l),1\}$. According to (\ref{sdfqq4}), (\ref{sdfqq5}) and Proposition~\ref{lemma-1234},
where $t_1 :=T-2lh, t_2 :=T-lh, t_3 :=T, S_0 :=C_0(A,\Omega,\widetilde{\omega}),
S_1 :=2\lambda\|B\|_{\infty}^{2}, y(t) :=\|\mathbf{f}(\cdot,t)\|^{2}, N(t) :=\mathbf{N}(t)$,
\begin{equation*}
F_{1}(t) :=\|a\|_{\infty}+\max_{1\leq i\leq2d}\|\nabla\Phi_{i}\|\|B\|_{\infty}+\frac{\lambda}{2}\|B\|^{2}_{\infty}
\end{equation*}
and
\begin{equation*}
F_{2}(t) :=\frac{C(A,\Omega,\widetilde{\omega})}{h^{2}}+2\left(\|a\|_{\infty}+\max_{1\leq i\leq2d}\|\nabla\Phi_{i}\|\|B\|_{\infty}\right)^{2},
\end{equation*}
it is clear that
\begin{eqnarray}\label{p88802Q}
\|\mathbf{f}(\cdot,T-lh)\|^{2(1+M)}\leq \|\mathbf{f}(\cdot,T)\|^{2}\|\mathbf{f}(\cdot,T-2lh)\|^{2M}K_{l,M}, \;\;\;\forall\; M\geq M_l.
\end{eqnarray}
Here,
\begin{equation}\label{12-17-26}
M_l:=3\frac{\int_{T-lh}^{T}\frac{e^{2\lambda\|B\|_\infty^2 t}}{(T-t+h)^{1+C_0(A,\Omega,\widetilde{\omega})}}\mathrm{d}t}{\int_{T-2lh}^{T-lh}\frac{e^{2\lambda\|B\|_\infty^2 t}}{(T-t+h)^{1+C_0(A,\Omega,\widetilde{\omega})}}\mathrm{d}t},
\end{equation}
\begin{eqnarray}\label{p88802Q222}
D_{l,M}:&=&12 l^2(M+1)C(A,\Omega,\widetilde{\omega})+ 12 lh(M+1)
\displaystyle{\int_{T-2lh}^{T}} \left(\|a\|_{\infty}+\displaystyle{\max_{1\leq i\leq2d}}\|\nabla\Phi_{i}\|\|B\|_{\infty}\right)^{2}\mathrm{d}t\\
&&+6lh(M+1)\|a\|_\infty+3\|B\|_\infty (M+1)\displaystyle{\int_{T-2lh}^{T}} (\max_{1\leq i\leq2d}\|\nabla\Phi_{i}\|)\mathrm{d}t
+15 \lambda l h(M+1)\|B\|_\infty^2,\nonumber
\end{eqnarray}
and
\begin{equation}\label{12-17-27}
K_{l,M}:=e^{D_{l,M}}(2l+1)^{3C_0(A,\Omega,\widetilde{\omega})(1+M)}.
\end{equation}
It follows from (\ref{p88802Q222}) and (\ref{12-17-27}) that
\begin{equation*}
D_{l,M}\leq (M+1)C(A,\Omega,\widetilde{\omega},l)\left(1+h^2\|a\|_\infty^2+\|B\|_\infty^2\right)
\end{equation*}
and
\begin{equation}\label{p88802m}
K_{l,M}\leq e^{(M+1)C(A,\Omega,\widetilde{\omega},l)\left(1+h^{2}\|a\|^{2}_{\infty}+\|B\|_{\infty}^{2}\right)}.
\end{equation}
This, along with (\ref{p88802Q}), implies (\ref{12-17-25}).\\

\textbf{Stage 5.} We end the proof of Theorem~\ref{jiudu4}.\\

It will be split into three steps as follows.\\

\textbf{Step 1.} We show that for $s:=s(A,\Omega,\widetilde{\omega})\in (0,1]$ sufficiently small, there are four constants $\widetilde{M}:=\widetilde{M}(A,\Omega,\widetilde{\omega})>0$, \color{black}
$\widetilde{C_1}:=\widetilde{C_1}(A,\Omega,\widetilde{\omega})>1$, $l_0:=l_0(A,\Omega,\widetilde{\omega})>1$ and $\mu:=\mu(\Omega,\widetilde{\omega})>0$
so that
\begin{equation}\label{AAAA}
\begin{array}{lll}
&&\displaystyle{\left(\int_{\Omega}|u(x,T)|^{2}\mathrm{d}x\right)^{1+\widetilde{M}}}\\
&\leq& e^{(1+\widetilde{M})\widetilde{C_1}\left[1+\|a\|_{\infty}^{2/3}+\|B\|_{\infty}^{2}+T\left(\|a\|_{\infty}
+\frac{\lambda}{2}\|B\|^{2}_{\infty}\right)\right]}\displaystyle{\left(\int_{\Omega} |u(x,0)|^{2}\mathrm{d}x\right)^{\widetilde{M}}}\\
&&\times \displaystyle{\left(e^{\frac{s\mu}{2h}}\int_{\widetilde{\omega}} |u(x,T)|^{2}dx+e^{-\frac{s\mu}{2h}}\int_{\Omega} |u(x,0)|^{2}\mathrm{d}x\right)},
\end{array}
\end{equation}
where $0<h\leq\min\{T/(4l_0), 1\}$, $\|a\|^{2/3}_{\infty}h<1$ and $l_{0}h\|B\|_{\infty}^2<1$.

To this end, firstly, on one hand, by Lemma~\ref{lemma-1.1}, we have that for any $0\leq t_1\leq t_2\leq T$,
\begin{eqnarray}\label{p88802km}
\|u(\cdot,t_2)\|\leq e^{(t_{2}-t_{1})\left(\|a\|_{\infty}+\frac{\lambda}{2}\|B\|^{2}_{\infty}\right)}\|u(\cdot,t_1)\|.
\end{eqnarray}
This yields that
\begin{equation}\label{12-22-1}
\begin{array}{lll}
\displaystyle{\int_{\Omega}} |u(x,T)|^{2}\mathrm{d}x&\leq& e^{2lh\left(\|a\|_{\infty}+\frac{\lambda}{2}\|B\|^{2}_{\infty}\right)}
\displaystyle{\int_{\Omega}} |u(x,T-lh)|^{2}e^{\frac{s}{(l+1)h}\varphi_{1,1}}e^{-\frac{s}{(1+l)h}\varphi_{1,1}}\mathrm{d}x\\
&\leq& e^{T\left(\|a\|_{\infty}+\frac{\lambda}{2}\|B\|^{2}_{\infty}\right)}e^{\frac{s C_1}{(l+1)h}}
\displaystyle{\int_{\Omega}} |f_{1}(x,T-lh)|^{2}\mathrm{d}x,
\end{array}
\end{equation}
where $l>1, 0<h\leq \min\{T/(4l),1\}$ and $C_1:=C_1(\Omega,\widetilde{\omega})>0$.
On the other hand, since $\varphi_{i,2}\leq \varphi_{i,1}$ in $\Omega$ for each $1\leq i\leq d$ (see (\ref{12-22-add})),
it follows from (\ref{p88802Q}) that for $s:=s(A,\Omega,\widetilde{\omega})\in (0,1]$ sufficiently small,
\begin{eqnarray}\label{p88802Qq}
\|f_{1}(\cdot,T-lh)\|^{2(1+M)}\leq \sum_{i=1}^{d}\|f_{i}(\cdot,T)\|^{2}
\left(\sum_{i=1}^{d}\|f_{i}(\cdot,T-2lh)\|^{2}\right)^{M}K_{l,M}, \;\;\;\forall\;M\geq M_l,
\end{eqnarray}
where $l>1$ and $0<h\leq\min\{T/(4l),1\}$. Moreover, noting that $\Phi_i\leq 0$ in
$\Omega\times [0,T]$ for each $1\leq i\leq d$ (see (\ref{12-22-add-2}) and (\ref{12-22-add})), by (\ref{p88802km}), we get that
\begin{eqnarray}\label{p88802Qqq}
\|f_{i}(\cdot,T-2lh)\|^{2}\leq e^{2T\left(\|a\|_{\infty}+\frac{\lambda}{2}\|B\|^{2}_{\infty}\right)}\|u(\cdot,0)\|^{2}.
\end{eqnarray}

Secondly, according to Proposition~\ref{lemma-absd1}, there exist two positive constants $r:=r(\Omega,\widetilde{\omega})$
and $\mu:=\mu(\Omega,\widetilde{\omega})\in (0,1)$ so that
\begin{equation*}
\omega_{i,r}:=\{x: |x-p_{i}|<r\}\subset \widetilde{\omega}\;\;\mbox{and}\;\;\varphi_{i,1}\leq -\mu\;\;\mbox{in}\;\;
\Omega\backslash\omega_{i,r},\;\;\;\forall\;1\leq i\leq d.
\end{equation*}
These, along with (\ref{12-22-add-2}) and Lemma~\ref{lemma-1.1}, imply that
\begin{equation}\label{p88802Qqqq}
\begin{array}{lll}
\|f_{i}(\cdot,T)\|^{2}&=& \displaystyle{\int_{\omega_{i,r}}} |u(x,T)|^{2}e^{\frac{s}{h}\varphi_{i,1}}\mathrm{d}x+\displaystyle{\int_{\Omega\backslash\omega_{i,r}}} |u(x,T)|^{2}e^{\frac{s}{h}\varphi_{i,1}}\mathrm{d}x\\
&\leq& \displaystyle{\int_{\widetilde{\omega}}} |u(x,T)|^{2}\mathrm{d}x+e^{-\frac{s\mu}{h}}e^{2T\left(\|a\|_{\infty}+\frac{\lambda}{2}\|B\|^{2}_{\infty}\right)}\displaystyle{\int_{\Omega}} |u(x,0)|^{2}\mathrm{d}x
\end{array}
\end{equation}
for each $1\leq i\leq d$. Here we used the fact that $\varphi_{i,1}\leq 0$  in $\Omega$ for each $1\leq i\leq d$ (see (\ref{12-22-add})).

Finally, it follows from  (\ref{12-22-1})-(\ref{p88802Qqqq}) that for $s:=s(A,\Omega,\widetilde{\omega})\in (0,1]$ sufficiently small,
\begin{eqnarray}\label{bbbb000}
&&\left(\int_{\Omega}|u(x,T)|^{2}\mathrm{d}x\right)^{1+M}\nonumber\\
&\leq&e^{(1+M)T\left(\|a\|_{\infty}+\frac{\lambda}{2}\|B\|^{2}_{\infty}\right)}e^{\frac{s C_1(1+M)}{(1+l)h}}\left(\int_{\Omega} |f_{1}(x,T-lh)|^{2}\mathrm{d}x\right)^{1+M}\nonumber\\
&\leq& K_{l,M}e^{(1+M)T\left(\|a\|_{\infty}+\frac{\lambda}{2}\|B\|^{2}_{\infty}\right)}e^{\frac{s C_1(1+M)}{(1+l)h}}\left(\sum_{i=1}^{d}\|f_{i}(\cdot,T-2lh)\|^{2}\right)^{M}
\left(\sum_{i=1}^{d}\|f_{i}(\cdot,T)\|^{2}\right)\\
&\leq& dK_{l,M}e^{(1+M)T\left(\|a\|_{\infty}+\frac{\lambda}{2}\|B\|^{2}_{\infty}\right)}e^{\frac{s C_1(1+M)}{(1+l)h}}
\left(de^{2T\left(\|a\|_{\infty}+\frac{\lambda}{2}\|B\|^{2}_{\infty}\right)}\int_{\Omega} |u(x,0)|^{2}\mathrm{d}x\right)^{M}\nonumber\\
&&\times
\left(\int_{\widetilde{\omega}} |u(x,T)|^{2}\mathrm{d}x+e^{-\frac{s\mu}{h}}e^{2T\left(\|a\|_{\infty}+\frac{\lambda}{2}\|B\|^{2}_{\infty}\right)}\int_{\Omega} |u(x,0)|^{2}\mathrm{d}x\right),\;\;\;\;\;\;\forall\;M\geq M_l.\nonumber
\end{eqnarray}
On one hand, when $0<h\leq\min\{ T/(4l), 1\}$ and $ lh\|B\|_{\infty}^2<1$, we obtain from (\ref{12-17-26}) that
\begin{equation}\label{12-22-3}
M_{l}\leq 3e^{4\lambda}\frac{(1+l)^{C_{0}(A, \Omega, \widetilde{\omega})}}{1-(\frac{2}{3})^{C_{0}(A, \Omega, \widetilde{\omega})}}:=\widetilde{M}_{l}.
\end{equation}
Obviously, there exists a constant $l_0:=l_0(A,\Omega,\widetilde{\omega})>1$ so that
\begin{eqnarray}\label{bbbb}
\frac{s C_1(1+\widetilde{M}_{l})}{(1+l)h}-\frac{s\mu}{h}\leq-\frac{s\mu}{2h}\;\;\;\mbox{for each}\;\;l\geq l_0.
\end{eqnarray}
Here, we used the fact that $C_0(A,\Omega,\widetilde{\omega})\in (0,1)$. On the other hand, when $0<h\leq\min\{T/(4l), 1\}$ and $\|a\|^{2/3}_{\infty}h<1$, we get from (\ref{p88802m}) that
\begin{eqnarray}\label{p88802m999}
K_{l,M}\leq e^{(M+1)C(A,\Omega, \widetilde{\omega}, l)\left(1+\|a\|^{2/3}_{\infty}+\|B\|_{\infty}^{2}\right)}.
\end{eqnarray}
It follows from (\ref{12-22-3})-(\ref{p88802m999}) that when $0<h\leq\min\{T/(4l_0), 1\}$,
$\|a\|^{2/3}_{\infty}h<1$ and  $l_{0}h\|B\|_{\infty}^2<1$,
\begin{eqnarray*}
M_{l_0}\leq \widetilde{M}_{l_0},\;\;\frac{s C_1(1+\widetilde{M}_{l_0})}{(1+l_0)h}-\frac{s\mu}{h}\leq-\frac{s\mu}{2h}
\;\;\mbox{and}\;\;K_{l_0,\widetilde{M}_{l_0}}\leq e^{C(A,\Omega, \widetilde{\omega})\left(1+\|a\|^{2/3}_{\infty}+\|B\|_{\infty}^{2}\right)}.
\end{eqnarray*}
These, along with (\ref{bbbb000}) (where $l$ and $M$ are replaced by $l_0$ and $\widetilde{M}_{l_0}$, respectively), imply (\ref{AAAA}).\\

\textbf{Step 2.} We show that for any  $h\in (0, 1]$ such that $h\geq T/(4l_{0})$ or $\|a\|^{2/3}_{\infty}h\geq1$ or $l_{0}h\|B\|_{\infty}^2\geq1$,
\begin{eqnarray}\label{CCCCC}
\int_\Omega |u(x,T)|^2\mathrm{d}x\leq e^{\widetilde{C_{2}}\left[\frac{1}{T}+\|a\|_{\infty}^{2/3}+\|B\|_{\infty}^{2}
+T\left(\|a\|_{\infty}+\frac{\lambda}{2}\|B\|^{2}_{\infty}\right)\right]}e^{-\frac{s\mu}{2h}}\int_{\Omega}|u(x,0)|^{2}
\mathrm{d}x
\end{eqnarray}
where $\widetilde{C_{2}}:=\widetilde{C_{2}}(A, \Omega, \widetilde{\omega})>0$ and $s\in (0,1]$.

Indeed, by (\ref{p88802km}), we can directly check that
\begin{equation}\label{12-23-3}
\begin{array}{lll}
\displaystyle{\int_\Omega} |u(x,T)|^2\mathrm{d}x&\leq& e^{2T\left(\|a\|_\infty+\frac{\lambda}{2}\|B\|_\infty^2\right)}\displaystyle{\int_\Omega} |u(x,0)|^2\mathrm{d}x\\
&\leq& e^{2T\left(\|a\|_\infty+\frac{\lambda}{2}\|B\|_\infty^2\right)}
e^{\frac{s\mu}{2}\left(\frac{4l_0}{T}+\|a\|_\infty^{2/3}+\|B\|_\infty^2\right)}e^{-\frac{s\mu}{2h}}
\displaystyle{\int_\Omega} |u(x,0)|^2\mathrm{d}x,
\end{array}
\end{equation}
where $h\in (0, 1]$ satisfies $h\geq T/(4l_{0})$ or $\|a\|^{2/3}_{\infty}h\geq1$ or $l_{0}h\|B\|_{\infty}^2\geq1$.
Then (\ref{CCCCC}) follows from (\ref{12-23-3}) immediately.\\

\textbf{Step 3.} We complete the proof of Theorem~\ref{jiudu4}.

By (\ref{AAAA}) and (\ref{CCCCC}), one can conclude that for $s:=s(A,\Omega,\widetilde{\omega})\in (0,1]$
sufficiently small and for any $h\in (0, 1]$,
\begin{eqnarray}\label{12-26-3}
\left(\int_{\Omega}|u(x,T)|^{2}\mathrm{d}x\right)^{1+\widetilde{M}}
&\leq& e^{(1+\widetilde{M})\widetilde{C_{3}}\left[1+\frac{1}{T}+\|a\|_{\infty}^{2/3}
+\|B\|_{\infty}^{2}+T\left(\|a\|_{\infty}+\frac{\lambda}{2}\|B\|^{2}_{\infty}\right)\right]}\left(\int_{\Omega} |u(x,0)|^{2}\mathrm{d}x\right)^{\widetilde{M}}\nonumber\\
&&\times \left(e^{\frac{s\mu}{2h}}\int_{\widetilde{\omega}} |u(x,T)|^{2}\mathrm{d}x+e^{-\frac{s\mu}{2h}}\int_{\Omega} |u(x,0)|^{2}\mathrm{d}x\right)
\end{eqnarray}
where $\widetilde{C_{3}}:=\widetilde{C_{3}}(A, \Omega, \widetilde{\omega})>2.$
Now, choose $h\in (0, 1]$ such that
\begin{eqnarray}\label{12-26-4}
&&\frac{1}{2}\left(\int_{\Omega}|u(x,T)|^{2}\mathrm{d}x\right)^{\widetilde{M}+1}\\
&=&e^{-\frac{s\mu}{2h}}e^{(1+\widetilde{M})\widetilde{C}_{3}\left[1+\frac{1}{T}+\|a\|_{\infty}^{2/3}
+\|B\|_{\infty}^{2}+T\left(\|a\|_{\infty}
+\frac{\lambda}{2}\|B\|^{2}_{\infty}\right)\right]}
\left(\int_{\Omega}|u(x,0)|^{2}\mathrm{d}x\right)^{\widetilde{M}+1}.\nonumber
\end{eqnarray}
It follows from (\ref{12-26-3}) and (\ref{12-26-4}) that
\begin{eqnarray*}
&&\left(\int_{\Omega}|u(x,T)|^{2}\mathrm{d}x\right)^{1+\widehat{M}}\\
&\leq& e^{(1+\widehat{M})\widetilde{C_{4}}\left[1+\frac{1}{T}+\|a\|_{\infty}^{2/3}+\|B\|_{\infty}^{2}
+T\left(\|a\|_{\infty}+\frac{\lambda}{2}\|B\|^{2}_{\infty}\right)\right]}
\left(\int_{\Omega}|u(x,0)|^{2}\mathrm{d}x\right)^{\widehat{M}}\int_{\widetilde{\omega}} |u(x,T)|^{2}\mathrm{d}x,
\end{eqnarray*}
where $\widehat{M}:=2\widetilde{M}+1$ and $\widetilde{C}_4:=\widetilde{C}_4(A,\Omega,\widetilde{\omega})>0$. Hence,
the result follows from the latter inequality immediately.\\

In summary, we finish the proof of Theorem~\ref{jiudu4}.\qed

\section{Proof of Theorem~\ref{Thm1}}\label{finalproof}

Now, we are able to present the proof of Theorem~\ref{Thm1} by using a telescoping series method.
For the convenience of the reader, we provide here the detailed
computation, although it is more or less similar to that of \cite[Theorem 4]{Phung-Wang-Zhang}.

\medskip

\noindent\textbf{Proof\ of\ Theorem~\ref{Thm1}}.
According to Theorem~\ref{jiudu4} and Young's inequality, it is clear that
for any $0\leq t_{1}<t_{2}\leq T$,
 \begin{equation}\label{2019-7-9}
 \|u(\cdot, t_{2})\|\leq\varepsilon
 \|u(\cdot, t_{1})\|+
 \frac{\mathcal{K}_{1}}{\varepsilon^{\alpha}}e^{\frac{\mathcal{K}_{2}}{t_{2}-t_{1}}}
 \|u(\cdot, t_{2})\|_{L^{2}(\widetilde{\omega})}, \ \ \ \forall \ \varepsilon>0.
 \end{equation}
Here, $\widetilde{\omega} \Subset\omega \subset\Omega$, $\mathcal{K}_{1}:= e^{\frac{\mathcal{K}}{2(1-\beta)}
{\left[1+\|a\|_{\infty}^{2/3}+\|B\|_{\infty}^{2}+T\left(\|a\|_{\infty}+\|B\|^{2}_{\infty}\right)\right]}}$,
$\mathcal{K}_{2}:= e^{\frac{\mathcal{K}}{2(1-\beta)}}$, and $\alpha:=\frac{\beta}{1-\beta}$
with $\mathcal{K}, \beta$ being the same constants appearing in Theorem~\ref{jiudu4}.
By Nash's inequality and Poincar$\mathrm{\acute{e}}$'s inequality,  we have that
\begin{equation}\label{2019-7-9090}
 \|g\|_{L^{2}(\widetilde{\omega})}\leq \mathcal{K}_{3}\|g\|^{\theta}_{L^{1}(\omega)}\|\nabla g\|^{1-\theta}, \ \ \ \forall \ g\in H_{0}^{1}(\Omega),
 \end{equation}
 where $\widetilde{\omega} \Subset\omega$, $\mathcal{K}_{3}:=\mathcal{K}_{3}(\Omega, \omega, \widetilde{\omega})>0$ and $\theta:=2/(2+N)$.
 Let $\eta\in C_{0}^{\infty}(\Omega)$ be such that $0\leq\eta\leq 1$ and $\eta\equiv1$ on $\widetilde{\omega}$.
 It follows from (\ref{2019-7-9090}) (where $g$ is replaced by $\eta u(\cdot,t_2)$) and Lemma~\ref{lemma-1.1} that
 \begin{eqnarray*}
 \|u(\cdot, t_{2})\|_{L^{2}(\widetilde{\omega})}&=& \|\eta u(\cdot,t_2)\|_{L^{2}(\widetilde{\omega})}\leq \mathcal{K}_{3}\|u(\cdot, t_{2})\|^{\theta}_{L^{1}(\omega)}\|\nabla (\eta u(\cdot,t_2))\|^{1-\theta}\\
 &\leq& \mathcal{K}_{3}\|u(\cdot, t_{2})\|^{\theta}_{L^{1}(\omega)}\left(\mathcal{K}_{4}\|u(\cdot, t_{2})\|+\|\nabla u(\cdot, t_{2})\|\right)^{1-\theta}\\
 &\leq& \mathcal{K}_{3}(\mathcal{K}_{4}\mathcal{K}_{5})^{1-\theta}\|u(\cdot, t_{2})\|^{\theta}_{L^{1}(\omega)}\left(\|u(\cdot, t_{1})\|+\|\nabla u(\cdot, t_{2})\|\right)^{1-\theta},
 \end{eqnarray*}
 where $\mathcal{K}_{4}:=\mathcal{K}_{4}(\Omega,\omega,\widetilde{\omega})>1$ and $\mathcal{K}_{5}:= e^{T\left(2\|a\|_{\infty}+\lambda\|B\|^{2}_{\infty}\right)
}.$ This, along with  Young's inequality, implies that
\begin{eqnarray}\label{2019-7-9090888}
 \|u(\cdot, t_{2})\|_{L^{2}(\widetilde{\omega})}\leq \delta\left(\|u(\cdot, t_{1})\|+\|\nabla u(\cdot, t_{2})\|\right)+\frac{\mathcal{K}_{6}}{\delta^{N/2}} \|u(\cdot, t_{2})\|_{L^{1}(\omega)}, \ \ \ \forall\ \delta>0.
 \end{eqnarray}
Here, $\mathcal{K}_{6}:=\mathcal{K}_3\left(\mathcal{K}_3 \mathcal{K}_4 \mathcal{K}_5\right)^{N/2}$. By Lemma~\ref{lemma-1.2}, we have that
\begin{eqnarray}\label{2019-7-9090999}
\|\nabla u(\cdot, t_{2})\|\leq\frac{\mathcal{K}_{7}}{(t_{2}-t_{1})^{1/2}} \|u(\cdot, t_{1})\|,
\end{eqnarray}
where $\mathcal{K}_{7}:=2\lambda^{3/2} e^{T\left[3\|a\|_{\infty}+\lambda\|B\|^{2}_{\infty}\right]}
.$
Hence, from (\ref{2019-7-9}), (\ref{2019-7-9090888}) and (\ref{2019-7-9090999}) with
\begin{eqnarray*}
&&\frac{\mathcal{K}_{1}}{\varepsilon^{\alpha}}e^{\frac{\mathcal{K}_{2}}{t_{2}-t_{1}}}\delta\left(1+\frac{\mathcal{K}_{7}}{(t_{2}-t_{1})^{1/2}}\right)
=\varepsilon,
\end{eqnarray*}
 we get that
\begin{equation}\label{2019-7-90901010}
\begin{array}{lll}
\displaystyle\|u(\cdot, t_{2})\|
&\leq&
\displaystyle 2\varepsilon
 \|u(\cdot, t_{1})\|+ \frac{\mathcal{K}_{1}\mathcal{K}_{6}}{\varepsilon^{\alpha}}e^{\frac{\mathcal{K}_{2}}{t_{2}-t_{1}}}\left[\frac{\mathcal{K}_{1}}{\varepsilon^{\alpha+1}}e^{\frac{\mathcal{K}_{2}}{t_{2}-t_{1}}}\left(1+\frac{\mathcal{K}_{7}}{(t_{2}-t_{1})^{1/2}}\right)\right]^{N/2} \|u(\cdot, t_{2})\|_{L^{1}(\omega)}\\
 &\leq&\displaystyle 2\varepsilon
 \|u(\cdot, t_{1})\|+ \frac{2^{N/2}\mathcal{K}^{1+N/2}_{1}\mathcal{K}_{6}\mathcal{K}^{N/2}_{7}}{\varepsilon^{\alpha+(\alpha+1)N/2}}e^{(\frac{3N}{4}+1)\frac{\mathcal{K}_{2}}{t_{2}-t_{1}}} \|u(\cdot, t_{2})\|_{L^{1}(\omega)}\\
 &\leq&\displaystyle 2\varepsilon
 \|u(\cdot, t_{1})\|+ \frac{\mathcal{K}_{8}}{(2\varepsilon)^{\gamma}}e^{\frac{\mathcal{K}_{9}}{t_{2}-t_{1}}} \|u(\cdot, t_{2})\|_{L^{1}(\omega)}, \ \ \forall \ \varepsilon>0,
\end{array}
\end{equation}
where $\gamma:=\alpha+(\alpha+1)N/2$, $\mathcal{K}_{8}:=2^{\alpha(1+N/2)+N}\mathcal{K}^{1+N/2}_{1}\mathcal{K}_{6}\mathcal{K}^{N/2}_{7}$ and $\mathcal{K}_{9}:=(1+3N/4)\mathcal{K}_{2}.$

On the other hand, let $E$ be a subset of positive measure in $(0,T)$.
Let $\ell_{0}$ be a density point of $E$. According to Proposition 2.1 in \cite{Phung-Wang},
for each $\kappa>1$, there exists $\ell_{1}\in (\ell_{0},T)$, depending on $\kappa$ and $E$,
so that the sequence $\{\ell_{m}\}_{m\geq1}$, given by
$$
\ell_{m+1}:=\ell_{0}+\frac{1}{\kappa^{m}}(\ell_{1}-\ell_{0}),
$$
satisfies that
 \begin{equation}\label{3.2525251}
\ell_{m}-\ell_{m+1}\leq 3|E\cap(\ell_{m+1},\ell_{m})|.
 \end{equation}
Next, let $0<\ell_{m+2}<\ell_{m+1}\leq t<\ell_{m}<\ell_{1}<T$. It follows from (\ref{2019-7-90901010}) that
\begin{equation}\label{3.2525252}
 \|u(\cdot, t)\|\leq\varepsilon
 \|u(\cdot, \ell_{m+2})\|+
 \frac{\mathcal{K}_{8}}{\varepsilon^{\gamma}}e^{\frac{\mathcal{K}_{9}}{t-\ell_{m+2}}}
 \|u(\cdot, t)\|_{L^{1}(\omega)}, \ \ \ \forall \ \varepsilon>0.
 \end{equation}
By Lemma~\ref{lemma-1.1}, we have that
\begin{eqnarray*}
\|u(\cdot, \ell_{m})\|\leq e^{T\left(\|a\|_{\infty}+\frac{\lambda}{2}\|B\|^{2}_{\infty}\right)}\|u(\cdot,t)\|.
\end{eqnarray*}
This, along with (\ref{3.2525252}), implies that
\begin{eqnarray*}
\|u(\cdot, \ell_{m})\|
\leq
e^{T\left(\|a\|_{\infty}+\frac{\lambda}{2}\|B\|^{2}_{\infty}\right)}
\left(\varepsilon
 \|u(\cdot, \ell_{m+2})\|+
 \frac{\mathcal{K}_{8}}{\varepsilon^{\gamma}}e^{\frac{\mathcal{K}_{9}}{t-\ell_{m+2}}}
 \|u(\cdot, t)\|_{L^{1}(\omega)}\right), \ \ \ \forall \ \varepsilon>0,
\end{eqnarray*}
which indicates that
$$\|u(\cdot, \ell_{m})\|
\leq \varepsilon
 \|u(\cdot, \ell_{m+2})\|+
 \frac{\mathcal{K}_{10}}{\varepsilon^{\gamma}}e^{\frac{\mathcal{K}_{9}}{t-\ell_{m+2}}}
 \|u(\cdot, t)\|_{L^{1}(\omega)}, \ \forall\ \varepsilon>0,
 $$
where $$\mathcal{K}_{10}:=\left[e^{T\left(\|a\|_{\infty}+\frac{\lambda}{2}\|B\|^{2}_{\infty}\right)}\right]^{1+\gamma}\mathcal{K}_{8}.$$
Integrating the latter inequality over $ E\cap(\ell_{m+1},\ell_{m})$, we get that
\begin{equation}\label{3.2525253}
\begin{array}{lll}
 \displaystyle{}|E\cap(\ell_{m+1},\ell_{m})|\|u(\cdot, \ell_{m})\|
 &\leq&\displaystyle{}\varepsilon |E\cap(\ell_{m+1},\ell_{m})|\|u(\cdot,\ell_{m+2})\|\\
 &&\displaystyle{}+\frac{\mathcal{K}_{10}}
 {\varepsilon^{\gamma}}e^{\frac{\mathcal{K}_{9}}{\ell_{m+1}-\ell_{m+2}}}
 \int_{\ell_{m+1}}^{\ell_{m}}\chi_{E}\|u(\cdot,t)\|_{L^{1}(\omega)}\mathrm{d}t,
 \ \forall\ \varepsilon>0.
\end{array}
\end{equation}
Here and in what follows, $\chi_{E}$ denotes the characteristic function of $E$.

Since $\ell_{m}-\ell_{m+1}=(\kappa-1)(\ell_{1}-\ell_{0})/\kappa^{m},$ by (\ref{3.2525253}) and (\ref{3.2525251}), we obtain that
\begin{eqnarray*}
\|u(\cdot,\ell_{m})\|&\leq& \varepsilon \|u(\cdot,\ell_{m+2})\|+\frac{1}{|E\cap(\ell_{m+1},\ell_{m})|}
\frac{\mathcal{K}_{10}}{\varepsilon^{\gamma}}
e^{\frac{\mathcal{K}_{9}}{\ell_{m+1}-\ell_{m+2}}}
\int_{\ell_{m+1}}^{\ell_{m}}\chi_{E}\|u(\cdot,t)\|_{L^{1}(\omega)}\mathrm{d}t\\
&\leq&\frac{3\kappa^{m}}{(\ell_{1}-\ell_{0})(\kappa-1)}
\frac{\mathcal{K}_{10}}{\varepsilon^{\gamma}}
e^{\mathcal{K}_{9}\left(\frac{1}{\ell_{1}-\ell_{0}}\frac{\kappa^{m+1}}{\kappa-1}\right)}
\int_{\ell_{m+1}}^{\ell_{m}}\chi_{E}\|u(\cdot,t)\|_{L^{1}(\omega)}\mathrm{d}t+
\varepsilon \|u(\cdot,\ell_{m+2})\|
 \end{eqnarray*}
 for each $\varepsilon>0$.
This yields that
\begin{equation}\label{3.2525254}
\begin{array}{lll}
 \displaystyle{}\|u(\cdot,\ell_{m})\|\leq \displaystyle{}
 \frac{1}{\varepsilon^{\gamma}}\frac{3}{\kappa}
 \frac{\mathcal{K}_{10}}{\mathcal{K}_{9}}
 e^{2\mathcal{K}_{9}\left(\frac{1}{\ell_{1}-\ell_{0}}
 \frac{\kappa^{m+1}}{\kappa-1}\right)}\int_{\ell_{m+1}}^{\ell_{m}}\chi_{E}
 \|u(\cdot,t)\|_{L^{1}(\omega)}\mathrm{d}t\displaystyle{}+
 \varepsilon \|u(\cdot,\ell_{m+2})\|
\end{array}
\end{equation}
for each $\varepsilon>0$.
Denote $\widetilde{d}:=\frac{2\mathcal{K}_{9}}{\kappa(\ell_{1}-\ell_{0})(\kappa-1)}$.
It follows from (\ref{3.2525254}) that
\begin{eqnarray*}
\varepsilon^{\gamma}e^{-\widetilde{d}\kappa^{m+2}}\|u(\cdot,\ell_{m})\|
-\varepsilon^{1+\gamma}e^{-\widetilde{d}\kappa^{m+2}}\|u(\cdot,\ell_{m+2})\|
\leq\frac{3}{\kappa}\frac{\mathcal{K}_{10}}
{\mathcal{K}_{9}}\int_{\ell_{m+1}}^{\ell_{m}}\chi_{E}\|u(\cdot,t)\|_{L^{1}(\omega)}\mathrm{d}t
\end{eqnarray*}
for each $\varepsilon>0$.
Choosing $\varepsilon=e^{-\widetilde{d}\kappa^{m+2}}$ in the latter inequality, we observe that
\begin{equation}\label{3.25252555}
\begin{array}{lll}
 \displaystyle{}e^{-(1+\gamma)\widetilde{d}\kappa^{m+2}}\|u(\cdot,\ell_{m})\|
 -e^{-(2+\gamma)\widetilde{d}\kappa^{m+2}}\|u(\cdot,\ell_{m+2})\|
 \leq\displaystyle{}\frac{3}{\kappa}\frac{\mathcal{K}_{10}}
 {\mathcal{K}_{9}}\int_{\ell_{m+1}}^{\ell_{m}}\chi_{E}\|u(\cdot,t)\|_{L^{1}(\omega)}\mathrm{d}t.
\end{array}
\end{equation}
Take $\kappa:=\sqrt{\frac{\gamma+2}{\gamma+1}}$ in (\ref{3.25252555}). Then we have that
\begin{eqnarray*}
e^{-(2+\gamma)\widetilde{d}\kappa^{m}}\|u(\cdot,\ell_{m})\|
-e^{-(2+\gamma)\widetilde{d}\kappa^{m+2}}\|u(\cdot,\ell_{m+2})\|
\leq \frac{3}{\kappa}\frac{\mathcal{K}_{10}}{\mathcal{K}_{9}}
\int_{\ell_{m+1}}^{\ell_{m}}\chi_{E}\|u(\cdot,t)\|_{L^{1}(\omega)}\mathrm{d}t.
\end{eqnarray*}
Changing $m$ to $2m'$ and summing the above inequality from $m'=1$ to infinity give the desired result. Indeed,
\begin{eqnarray*}
&&e^{-T\left(\|a\|_{\infty}+\frac{\lambda}{2}\|B\|^{2}_{\infty}\right)}e^{-(2+\gamma)\widetilde{d}\kappa^{2}}\|u(\cdot,T)\|\\
&\leq& e^{-(2+\gamma)\widetilde{d}\kappa^{2}}\|u(\cdot,\ell_{2})\|\\
&\leq&\sum_{m'=1}^{+\infty}\left(e^{-(2+\gamma)\widetilde{d}\kappa^{2m'}}\|u(\cdot,\ell_{2m'})\|
-e^{-(2+\gamma)\widetilde{d}\kappa^{2m'+2}}\|u(\cdot,\ell_{2m'+2})\|\right)\\
&\leq& \frac{3}{\kappa}\frac{\mathcal{K}_{10}}{\mathcal{K}_{9}}\sum_{m'=1}^{+\infty}
\int_{\ell_{2m'+1}}^{\ell_{2m'}}\chi_{E}\|u(\cdot,t)\|_{L^{1}(\omega)}\mathrm{d}t\\
&\leq& \frac{3}{\kappa}\frac{\mathcal{K}_{10}}
{\mathcal{K}_{9}}\int_{0}^{T}\chi_{E}\|u(\cdot,t)\|_{L^{1}(\omega)}\mathrm{d}t.
 \end{eqnarray*}

In summary, we finish the proof of Theorem~\ref{Thm1}.
\qed

\section{Appendix}\label{app}
\noindent\textbf{Proof\ of\ Lemma~\ref{lemma-1.1}}. Multiplying the first equation of \eqref{1.1} by $u$ and integrating it over $\Omega\times(0,t)$ (with any $t\in (0,T]$), we have that
$$\int_{\Omega}|u(x,t)|^{2}\mathrm{d}x-\int_{\Omega}|u(x,0)|^{2}\mathrm{d}x+2\int_{0}^{t}\int_{\Omega}A\nabla u\cdot\nabla u\mathrm{d}x\mathrm{d}s=-2\int_{0}^{t}\int_{\Omega}B\cdot\nabla u u\mathrm{d}x\mathrm{d}s-2\int_{0}^{t}\int_{\Omega}a u^{2}\mathrm{d}x\mathrm{d}s.$$
By (\ref{yu-11-28-2}) and Young's inequality, we get that

\begin{eqnarray*}
&&\int_{\Omega}|u(x,t)|^{2}\mathrm{d}x+2\lambda^{-1}\int_{0}^{t}\int_{\Omega}|\nabla u|^{2}\mathrm{d}x\mathrm{d}s\\
&\leq&\int_{\Omega}|u(x,0)|^{2}\mathrm{d}x+2\|B\|_{\infty}\int_{0}^{t}\int_{\Omega}|\nabla u||u|\mathrm{d}x\mathrm{d}s+2\|a\|_{\infty}\int_{0}^{t}\int_{\Omega}|u|^{2}\mathrm{d}x\mathrm{d}s\\
&\leq&\int_{\Omega}|u(x,0)|^{2}\mathrm{d}x+\lambda^{-1}\int_{0}^{t}\int_{\Omega}|\nabla u|^{2}\mathrm{d}x\mathrm{d}s+\left(2\|a\|_{\infty}+\lambda\|B\|^{2}_{\infty}\right)\int_{0}^{t}\int_{\Omega}|u|^{2}\mathrm{d}x\mathrm{d}s,
\end{eqnarray*}
which indicates
\begin{eqnarray*}
&&\int_{\Omega}|u(x,t)|^{2}\mathrm{d}x+\lambda^{-1}\int_{0}^{t}\int_{\Omega}|\nabla u|^{2}\mathrm{d}x\mathrm{d}s\\
&\leq&\int_{\Omega}|u(x,0)|^{2}\mathrm{d}x+\left(2\|a\|_{\infty}+\lambda\|B\|^{2}_{\infty}\right)\int_{0}^{t}\int_{\Omega}|u|^{2}\mathrm{d}x\mathrm{d}s\\
&\leq&\int_{\Omega}|u(x,0)|^{2}\mathrm{d}x+\left(2\|a\|_{\infty}+\lambda\|B\|^{2}_{\infty}\right)\int_{0}^{t}\left(\int_{\Omega}|u(x,s)|^{2}\mathrm{d}x+\lambda^{-1}\int_{0}^{s}\int_{\Omega}|\nabla u(x,\tau)|^{2}\mathrm{d}x\mathrm{d}\tau\right)\mathrm{d}s.
\end{eqnarray*}
This, along with Gronwall's inequality, implies  \eqref{lemma-1.1}.\qed
\\

\noindent\textbf{Proof\ of\ Lemma~\ref{lemma-1.2}}. Multiplying the first equation of \eqref{1.1} by  $-t \mathrm{div}(A\nabla u)$ and integrating it over $\Omega\times(0,t)$, we obtain that
\begin{eqnarray*}
&&t\int_{\Omega}A\nabla u(x,t)\cdot\nabla u(x,t)\mathrm{d}x+2\int_{0}^{t}\int_{\Omega}s[\mathrm{div}(A\nabla u)]^{2}\mathrm{d}x\mathrm{d}s-\int_{0}^{t}\int_{\Omega}A\nabla u\cdot\nabla u\mathrm{d}x\mathrm{d}s\\
&=&2\int_{0}^{t}\int_{\Omega}s B\cdot\nabla u\mathrm{div}(A\nabla u)\mathrm{d}x\mathrm{d}s+2\int_{0}^{t}\int_{\Omega}s au\mathrm{div}(A\nabla u)\mathrm{d}x\mathrm{d}s.
\end{eqnarray*}
This, along with Young's inequality, implies that

\begin{eqnarray}\label{1.222222111}
&&t\int_{\Omega}A\nabla u(x,t)\cdot\nabla u(x,t)\mathrm{d}x-\int_{0}^{t}\int_{\Omega}A\nabla u\cdot\nabla u\mathrm{d}x\mathrm{d}s\nonumber\\
&\leq&\|B\|^{2}_{\infty}\int_{0}^{t}\int_{\Omega}s |\nabla u|^{2}\mathrm{d}x\mathrm{d}s+\|a\|^{2}_{\infty}\int_{0}^{t}\int_{\Omega}s |u|^{2}\mathrm{d}x\mathrm{d}s.
\end{eqnarray}
By (\ref{yu-11-28-2}) and (\ref{1.2}), we get that

\begin{eqnarray}\label{1.222222}
\int_{0}^{t}\int_{\Omega}A\nabla u\cdot\nabla u\mathrm{d}x\mathrm{d}s\leq \lambda^{2} e^{t\left(2\|a\|_{\infty}+\lambda\|B\|^{2}_{\infty}\right)}\int_{\Omega}|u(x,0)|^{2}\mathrm{d}x
\end{eqnarray}
and
\begin{equation}\label{1.222223}
\begin{array}{lll}
\displaystyle\|a\|^{2}_{\infty}\int_{0}^{t}\int_{\Omega}s |u|^{2}\mathrm{d}x\mathrm{d}s
&\leq&\displaystyle \|a\|^{2}_{\infty}\int_{0}^{t}se^{s\left(2\|a\|_{\infty}+\lambda\|B\|^{2}_{\infty}\right)}\mathrm{d}s\int_{\Omega}|u(x,0)|^{2}\mathrm{d}x\\
&\leq&\displaystyle \|a\|_{\infty}\int_{0}^{t}e^{s\left(3\|a\|_{\infty}+\lambda\|B\|^{2}_{\infty}\right)}\mathrm{d}s\int_{\Omega}|u(x,0)|^{2}\mathrm{d}x\\
&\leq&\displaystyle e^{t\left(3\|a\|_{\infty}+\lambda\|B\|^{2}_{\infty}\right)}\int_{\Omega}|u(x,0)|^{2}\mathrm{d}x.
\end{array}
\end{equation}
It follows from (\ref{yu-11-28-2}) and (\ref{1.222222111})-(\ref{1.222223}) that
\begin{eqnarray*}
t\int_{\Omega}|\nabla u(x,t)|^{2}\mathrm{d}x\leq
2\lambda^{3} e^{t(3\|a\|_{\infty}+\lambda\|B\|^{2}_{\infty})}\int_{\Omega}|u(x,0)|^{2}\mathrm{d}x+\lambda\|B\|^{2}_{\infty}\int_{0}^{t}s\int_{\Omega} |\nabla u|^{2}\mathrm{d}x\mathrm{d}s,
\end{eqnarray*}
which, combined with Gronwall's inequality, indicates (\ref{1.3}).\qed

\end{document}